\documentclass[10pt,a4paper]{article}
\usepackage{amsthm}
\usepackage{amsmath}
\usepackage{amsfonts}
\usepackage{amssymb}
\usepackage{graphicx}
\usepackage[utf8]{inputenc}
\usepackage{verbatim}
\usepackage{comment}
\usepackage{subcaption}
\usepackage{float}
\usepackage{mathtools}
\usepackage[scaled=0.90]{helvet}
\usepackage{hyperref}
\usepackage{geometry}
\usepackage{pgf,tikz,pgfplots}
\pgfplotsset{compat=1.15}
\usetikzlibrary{calc, arrows, shapes, positioning, patterns, angles, quotes, intersections, through, backgrounds}\usetikzlibrary{decorations.markings}

\makeatletter
\let\save@mathaccent\mathaccent
\newcommand*\if@single[3]{%
  \setbox0\hbox{${\mathaccent"0362{#1}}^H$}%
  \setbox2\hbox{${\mathaccent"0362{\kern0pt#1}}^H$}%
  \ifdim\ht0=\ht2 #3\else #2\fi
  }
\newcommand*\rel@kern[1]{\kern#1\dimexpr\macc@kerna}
\newcommand*\widebar[1]{\@ifnextchar^{{\wide@bar{#1}{0}}}{\wide@bar{#1}{1}}}
\newcommand*\wide@bar[2]{\if@single{#1}{\wide@bar@{#1}{#2}{1}}{\wide@bar@{#1}{#2}{2}}}
\newcommand*\wide@bar@[3]{%
  \begingroup
  \def\mathaccent##1##2{%
    \let\mathaccent\save@mathaccent
    \if#32 \let\macc@nucleus\first@char \fi
    \setbox\z@\hbox{$\macc@style{\macc@nucleus}_{}$}%
    \setbox\tw@\hbox{$\macc@style{\macc@nucleus}{}_{}$}%
    \dimen@\wd\tw@
    \advance\dimen@-\wd\z@
    \divide\dimen@ 3
    \@tempdima\wd\tw@
    \advance\@tempdima-\scriptspace
    \divide\@tempdima 10
    \advance\dimen@-\@tempdima
    \ifdim\dimen@>\z@ \dimen@0pt\fi
    \rel@kern{0.6}\kern-\dimen@
    \if#31
      \overline{\rel@kern{-0.6}\kern\dimen@\macc@nucleus\rel@kern{0.4}\kern\dimen@}%
      \advance\dimen@0.4\dimexpr\macc@kerna
      \let\final@kern#2%
      \ifdim\dimen@<\z@ \let\final@kern1\fi
      \if\final@kern1 \kern-\dimen@\fi
    \else
      \overline{\rel@kern{-0.6}\kern\dimen@#1}%
    \fi
  }%
  \macc@depth\@ne
  \let\math@bgroup\@empty \let\math@egroup\macc@set@skewchar
  \mathsurround\z@ \frozen@everymath{\mathgroup\macc@group\relax}%
  \macc@set@skewchar\relax
  \let\mathaccentV\macc@nested@a
  \if#31
    \macc@nested@a\relax111{#1}%
  \else
    \def\gobble@till@marker##1\endmarker{}%
    \futurelet\first@char\gobble@till@marker#1\endmarker
    \ifcat\noexpand\first@char A\else
      \def\first@char{}%
    \fi
    \macc@nested@a\relax111{\first@char}%
  \fi
  \endgroup
}
\makeatother

\usepackage{capt-of}

\newcommand*{\N}{\mathbb{N}}

\newcommand*{\R}{\mathbb{R}}

\newcommand*{\mI}{\mathcal{I}}

\newcommand{\Tau}{\mathcal{T}}

\newcommand*{\Sph}{\mathbb{S}}

\newcommand*{\td}{\widetilde{d}}
\newcommand*{\rar}{\rightarrow}

\newcommand*{\lpls}{(X, d, \ll, \leq, \tau)}

\newcommand*{\tll}{\mathrel{\widetilde{\ll}}}
\newcommand*{\tleq}{\mathrel{\widetilde{\leq}}}
\newcommand*{\tl}{\mathrel{\widetilde{<}}}

\newcommand*{\ttau}{\widetilde{\tau}}

\newcommand*{\amA}{X_1 \sqcup_A X_2}

\swapnumbers

\newtheorem{thm}{Theorem}[section]
\newtheorem{prop}[thm]{Proposition}
\newtheorem{cor}[thm]{Corollary}
\newtheorem{lem}[thm]{Lemma}
\newtheorem{defin}[thm]{Definition}

\theoremstyle{definition}
\newtheorem{ex}[thm]{Example}

\newtheorem{rem}[thm]{Remark}
\numberwithin{equation}{section}

\theoremstyle{definition} 

\newcommand{\thistheoremnam}{}
\newtheorem*{genericthm*}{\thistheoremnam}
\newenvironment{chapt*}[1]
  {\renewcommand{\thistheoremnam}{#1}%
   \begin{genericthm*}}
  {\end{genericthm*}}

\title{Gluing of Lorentzian length spaces and the causal ladder}
\author{Felix Rott\footnote{\href{mailto:felix.rott@univie.ac.at}{felix.rott@univie.ac.at}, Faculty of Mathematics, University of Vienna, Austria.}}
\date{}

\begin{document}

\maketitle

\begin{abstract}
We investigate the compatibility of Lorentzian amalgamation
with various properties
of Lorentzian pre-length spaces. In particular, we give conditions under
which gluing of
Lorentzian length spaces yields again a Lorentzian length space and we
give criteria which
preserve many steps of the causal ladder.
We conclude with some thoughts on the causal properties which seem not
so easily transferable. \\

\emph{Keywords:} Lorentzian length spaces, gluing constructions, quotient spaces, metric geometry, causality theory \\

\emph{MSC2020:} 53C23 (primary), 53C50, 53B30, 51F99, 51K10 (secondary)
\end{abstract}

\section{Introduction}
Lorentzian length spaces are a comparatively new approach to a synthetic
description of Lorentzian geometry, introduced in \cite{KS18}.
This is in spirit very similar to metric geometry and the theory of
length spaces, which was key to developing a metric and synthetic point
of view of Riemannian geometry, without relying on any differential
structure and machinery. \\

With the paper birthing the theory being only a couple of years old,
this theory is, by mathematical standards, still in its infancy.
There is, however, rapid progress being made in various directions.
The advancements in the theory of Lorentzian length spaces can roughly
be sorted into two directions (which are not disjoint, of course): on the one hand, there is the translation of fundamental and almost elementary concepts originally developed for
(metric) length spaces aiming to improve the robustness of Lorentzian
(pre-)length spaces and bring it to a level of sophistication close to
that of the metric counterpart. These range from the introduction of (hyperbolic) angles to the description of 
Ricci curvature bounds, Hausdorff measure and
gluing. Works in this direction include:

\begin{itemize}
\item[(i)] \cite{CM20} introduces optimal transport methods in Lorentzian length spaces, defines timelike
Ricci curvature bounds via suitable entropy conditions and gives applications to general
relativity (synthetic singularity theorems).

\item[(ii)] \cite{KS21} examines the null distance in Lorentzian length spaces (which was first introduced
in \cite{SV16} for spacetimes) and in turn studies Gromov-Hausdorff convergence, establishing
first compatibility results with respect to curvature bounds.

\item[(iii)] \cite{MS21} defines an analogue to Hausdorff measure and dimension on Lorentzian length spaces.

\item[(iv)] \cite{BR22} introduces gluing techniques for Lorentzian pre-length spaces and gives an analogue to the gluing theorem of Reshetnyak for CAT($k$) spaces. 

\item[(v)] \cite{BS22} introduces the concept of (hyperbolic) angles, a notion very
similar to the classical Alexandrov angle in metric geometry, cf. \cite[Definition I.1.12]{BH99},
as well as timelike tangent cones and an exponential map, which mimic
the corresponding notions in a metric (length) space, cf. \cite[Definitions II.3.18 \& II.4.4]{BH99}. Via
hyperbolic angles, the authors also give a formulation of curvature
bounds in terms of angle monotonicity.

\item[(vi)] \cite{BMdOS22} also developed, parallel to the above work \cite{BS22}, hyperbolic angles on Lorentzian length spaces. In addition to a monotonicity
formulation of curvature bounds, the authors also introduce a formulation
relating the size of angles and their comparison angles, similar to
the classical angle condition in metric geometry, cf.\ \cite[Definition 4.1.5]{BBI}.
\end{itemize}

On the other hand, there are attempts to give synthetic versions of
various ideas from classical Lorentzian geometry and causality theory. 
Works in this direction include:
\begin{itemize}
\item[(i)] \cite{GKS19} develops a notion of (in)extendibility for Lorentzian length spaces.

\item[(ii)] \cite{AGKS19} defines an analogue to warped products for Lorentzian length spaces and relates the timelike curvature of the product to the (metric) curvature of the base space.

\item[(iii)] \cite{ACS20} expands the causal ladder for Lorentzian length spaces which was originally introduced in \cite{KS18}.

\item[(iv)] \cite{BGH21} introduces time functions on Lorentzian length spaces and shows their existence to being equivalent to $K$-causality.
\end{itemize}
Simply put, one is interested in ``Lorentzifying'' known synthetic concepts
from metric geometry as well as synthesizing relativistic concepts
from Lorentzian geometry.
This article should foremost be regarded as a follow-up to \cite{BR22}, so
at first glance appears to be placed more into the first category.
However, we are mostly investigating how notions from causality theory
behave with respect to gluing, so in essence it is a mix of the two
(which is also a valid description of many works in the theory of
Lorentzian length spaces). \\

We begin by recalling some basic results and definitions about
Lorentzian pre-length spaces, gluing and the causal ladder. We continue
with establishing under which assumptions the additional properties of a
Lorentzian length space are retained. Then we discuss the inheritance of
most of the steps of the causal ladder. We will usually denote by $X$ the amalgamation of two Lorentzian pre-length spaces $X_1$ and $X_2$ and assume that the identifying map preserves the Lorentzian structure (more details on this below). \\

Finally, we collect the main results of this work in the following two theorems (see below for definitions of the required notions):

\begin{thm}[Causal inheritance]
Let $X_1$ and $X_2$ be two Lorentzian pre-length spaces and $X:= \amA$ their Lorentzian amalgamation. Then we have the following preservation of causality conditions.
\begin{itemize}
\item[(i)] If $X_1$ and $X_2$ are chronological, then so is $X$.

\item[(ii)] If $X_1$ and $X_2$ are causal, then so is $X$.

\item[(iii)] If $X_1$ and $X_2$ are extrinsically non-totally imprisoning, then $X$ is non-totally imprisoning.

\item[(iv)] If $X_1$ and $X_2$ are strongly causal, then so is $X$.

\item[(iv)] If $X_1$ and $X_2$ are strongly causal, non-timelike locally isolating and distinguishing, then $X$ is distinguishing.

\item[(v)] If $X_1$ and $X_2$ are causally path-connected, locally causally closed, d-compatible and globally hyperbolic with $A_1$ and $A_2$ time observing, then $X$ is globally hyperbolic.
\end{itemize}
\end{thm}
\begin{thm}[Amalgamation of length spaces]
Let $X_1$ and $X_2$ be two strongly causal and locally compact Lorentzian length spaces. Then $X:= \amA$ is a (strongly localizable) Lorentzian length space.
\end{thm}

\section{Preliminaries}
Here, we briefly collect some essentials regarding the general theory of Lorentzian pre-length spaces as well as Lorentzian amalgamation and the causal ladder. For more details, see \cite{KS18, BR22, ACS20}. As in \cite{BR22}, we set $\inf \emptyset = \infty$ and $\sup \emptyset = 0$ for convenience, as time separation and distance metric take values only in $[0,\infty]$. Moreover, if $\alpha:[a,b] \to X, \beta:[b,c] \to X$ are two curves, we will denote by $\alpha \ast \beta:[a,c] \to X$ their concatenation.

\begin{defin}[Lorentzian pre-length space]
A tuple $\lpls$ is called a Lorentzian pre-length space if the following holds:
\begin{itemize}
\item[(i)] $(X,\ll,\leq)$ is a causal space, i.e., $\leq$ is a reflexive and transitive relation on $X$ and $\ll$ is a transitive relation contained in $\leq$.
\item[(ii)] $\tau: X \times X \to [0,\infty]$ is lower semi-continuous with respect to the metric $d$.
\item[(iii)] If $x \leq y \leq z$, then $\tau(x,z) \geq \tau (x,y) + \tau(y,z)$, and $\tau(x,y) > 0 \iff x \ll y$.
\end{itemize}
\end{defin}
By \cite[Example 2.11]{KS18}, any smooth spacetime is a Lorentzian pre-length space (where the distance metric $d$ is given by a (complete) Riemannian metric). In fact, any continuous and causally plain spacetime is a Lorentzian pre-length space as well, see \cite[Proposition 5.8]{KS18}. \\

We will usually abbreviate a Lorentzian pre-length space $\lpls$ by $X$. 
By $x < y$ we mean $x \leq y$ and $x \neq y$ and we will use common notation for sets described via causality relations, e.g., $J^+(x):=\{y \in X \mid x \leq y\}$ and $I(x,z):= \{y \in X \mid x \ll y \ll z\}=I^+(x) \cap I^-(z)$. \\

In contrast to a metric length space, the correct notion of an ``intrinsic Lorentzian space'' is a bit more complicated.
To get from a Lorentzian pre-length space to a Lorentzian length space, one first needs to introduce the concept of causal and timelike curves.

\begin{defin}[Causal/timelike curves]
Let $\lpls$ be a Lorentzian pre-length space. 
\begin{itemize}
\item[(i)] A locally Lipschitz curve $\gamma:[a,b] \to X$ is called future-directed causal (respectively timelike), if $\gamma(s) \leq \gamma(t)$ (respectively $\gamma(s) \ll \gamma(t)$) for all $s,t \in [a,b], s < t$.
Past-directed curves are defined analogously.
Unless explicitly stated otherwise, we assume all causal curves to be future-directed.
\item[(ii)] The $\tau$-length of a causal curve $\gamma$ is given as
\begin{equation}
L_{\tau}(\gamma):= \inf \{ \sum_{i=0}^n \tau(\gamma(t_i),\gamma(t_{i+1}) \mid 
a=t_0 < t_1 < \ldots < t_n=b, n \in \N \}.
\end{equation}
If $\gamma(a)=x,\gamma(b)=y$ and $L_{\tau}(\gamma)=\tau(x,y)$ we say that $\gamma$ is $\tau$-realizing and we call (the image of) such a curve a geodesic (segment).
\end{itemize}
\end{defin}

The next ingredient is localizability, a concept which is similar to the existence of small convex neighbourhoods in Lorentzian manifolds.

\begin{defin}[Localizability]
\label{def: loc}
Let $X$ be a Lorentzian pre-length space. $X$ is called localizable if every point in X has a neighbourhood $U$ such that the following holds:
\begin{itemize}
\item[(i)] There is a constant $C > 0$ such that $L^d(\gamma) 
\footnote{Recall that the length of a curve $\gamma : [a,b] \to X$ in a metric space is given as $L_d(\gamma):=\sup \{ \sum_{i=0}^{n-1} d(\gamma(t_i), \gamma(t_{i+1})) \mid a=t_0 < t_1 < \ldots < t_n=b, n \in \N \}$.} 
\leq C$ for all causal curves $\gamma$ contained in $U$. 

\item[(ii)] There is a continuous map $\omega=\omega_U: U \times U \to [0,\infty)$ such that $(U,d|_{U \times U}, {\ll}|_{U \times U}, {\leq}|_{U \times U}, \omega)$ is a Lorentzian pre-length space with the following non-triviality condition: $I^{\pm}(p) \cap U \neq \emptyset$ for all $p \in U$.

\item[(iii)] For all $p,q \in U$ with $p \leq q$ there is a causal curve $\gamma$ from $p$ to $q$ contained in $U$ which is maximal in $U$ and satisfies $L_{\tau}(\gamma) = \omega(p,q)$.
\end{itemize}
Note that $\omega$ has necessarily the explicit form
\begin{equation}
\omega(p,q)= \max \{L_{\tau}(\gamma) \mid \gamma \text{ is a causal curve from $p$ to $q$ contained in $U$}\}.
\end{equation}
$U$ is called a localizable neighbourhood. 
If every point has a neighbourhood basis of localizable neighbourhoods, $X$ is called strongly localizable.
\end{defin}

The following definition contains some elementary properties of Lorentzian pre-length spaces, some of which are required for a Lorentzian length space.

\begin{defin}[Further properties of Lorentzian pre-length spaces]
\label{def: basic prop}
Let $X$ be a Lorentzian pre-length space.
\begin{itemize}
\item[(i)] $X$ is called causally path-connected if $x \ll y$ implies that there exists a timelike curve from $x$ to $y$ and $x \leq y$ implies there is a causal curve from $x$ to $y$.

\item[(ii)] Let $x \in X$ and let $U \subseteq X$ be a neighbourhood of $x$. $U$ is called causally closed if ${\leq}|_{\widebar{U} \times \widebar{U}}$ is closed, i.e., for all sequences $(p_n)_{n \in \N}, (q_n)_{n \in \N}$ in $U$ with $p_n \leq q_n$ for all $n \in \N$ and $p_n \to p \in \widebar{U}, q_n \to q \in \widebar{U}$ we have $p \leq q$.

\item[(iii)] $X$ is called locally causally closed if every point has a causally closed neighbourhood.

\item[(iv)] Define the function $\Tau: X \times X \to [0,\infty]$, 
\begin{equation}
\Tau(x,y):=\sup \{L_{\tau}(\gamma) \mid \gamma \text{ is a future-directed causal curve from $x$ to $y$}\}.
\end{equation}
We say $X$ is intrinsic if $\tau=\Tau$. 

\item[(v)] $X$ is called $d$-compatible if for all $x \in X$ there is a neighbourhood $U \subseteq X$ and a positive constant $C$ such that $L_d(\gamma) \leq C$ for all causal curves $\gamma$ contained in $U$.

\item[(vi)] A subset $A$ of $X$ is called causally convex if $J(p,q) \subseteq A$ for all $p,q \in A$\footnote{By definition, $I^{\pm}(x), J^{\pm}(x)$ are causally convex for all $x \in X$. 
Moreover, intersections of causally convex sets are causally convex and hence in particular causal and timelike diamonds are causally convex. 
Note that if $X$ is causally path-connected, then causal convexity can be formulated as in the case of spacetimes, namely all causal curves between points in $A$ are entirely contained in $A$.}.

\item[(vii)] $X$ is called interpolative if for all $x,z \in X$ with $x \ll z$ there exists $y \in X$ such that $x \ll y \ll z$.
\end{itemize}
\end{defin}

\begin{defin}[Lorentzian length space]
A locally causally closed, causally path-connected, localizable and intrinsic Lorentzian pre-length space is called a Lorentzian length space.
\end{defin}

When dealing with metric spaces or even topological spaces, the notion of a subspace is of great importance and hence also desired in our setting.

\begin{defin}[Lorentzian subspace]
\label{def: subspace}
Let $A$ be a subset of a Lorentzian pre-length space $\lpls$. There is a natural way to view $A$ as a subspace, namely restricting the original structure of $X$ to $A$, i.e., the subspace $A$ is the Lorentzian pre-length space $(A,d|_{A \times A}, {\ll}|_{A \times A}, \\
{\leq}|_{A \times A}, \tau|_{A \times A})$.
\end{defin}

As in the metric case, a subspace of a Lorentzian length space is in general not a Lorentzian length space anymore. 
The next definition introduces the key property required of subsets along which one wants to glue.

\begin{defin}[Timelike isolation]
Let $X$ be a Lorentzian pre-length space.
\begin{itemize}
\item[(i)] $X$ is said to contain no $\ll$-isolated points if $I^{\pm}(x) \neq \emptyset$ for all $x \in X$.
\item[(ii)] A subset $A \subseteq X$ is said to be non-future locally isolating if for all $a \in A$ with $I^+(a) \neq \emptyset$ and for all neighbourhoods $U$ of $a$ there exists $b_+ \in U \cap A$ such that $a \ll b_+$. Similarly, we define a non-past locally isolating set. We say $A$ is non-timelike locally isolating if it satisfies both properties.
\end{itemize}
Note that $X$ (as a subset of itself) is non-timelike locally isolating if and only if no open subset of $X$ (viewed as a subspace) contains any $\ll$-isolated points. 
\end{defin}

Amalgamation in the metric case consists of two steps: first forming the disjoint union and then forming the quotient with respect to an equivalence relation which results from the identification of distinguished subsets. The concept of a disjoint union construction can be adapted naturally. For more details on amalgamation in the metric case see \cite{BBI, BH99}.

\begin{defin}[Lorentzian disjoint union]
Let $(X_1,d_1,\ll_1,\leq_1,\tau_1)$ and $(X_2,d_2,\ll_2,\leq_2,\tau_2)$ be two Lorentzian pre-length spaces and set $X:= X_1 \sqcup X_2$. Define ${\leq} := {\leq_1} \sqcup {\leq_2}$, i.e., ``${\leq} \subseteq X \times X$'' and $x \leq y :\iff \exists i \in \{1,2\}: x,y \in X_i \wedge x \leq_i y$. Similarly, define ${\ll} := {\ll_1} \sqcup {\ll_2}$. Let $d$ be the disjoint union metric on $X$,  
i.e., 
\begin{equation}
d(x,y)=
\begin{cases} d_i(x,y) & x,y \in X_i \\
\infty & \, \text{else}.
\end{cases}
\end{equation}
Define $\tau:X \times X \rar [0,\infty]$ by 
\begin{equation}
\tau(x,y):=
\begin{cases} \tau_i(x,y) & x,y \in X_i \\
0 & \, \text{else}.
\end{cases}
\end{equation}
We call $\lpls$ the Lorentzian disjoint union of $X_1$ and $X_2$.
\end{defin}

It is easily seen that this always results in a Lorentzian pre-length space, cf.\cite[Proposition 3.2.2]{BR22}.
This is in some contrast to the following quotient construction. While the map introduced below is well-defined, it may happen that it is not lower semi-continuous. For a counterexample see \cite[Example 3.1.8]{BR22}.

\begin{defin}[Quotient Lorentzian structure]
Let $X$ be a Lorentzian pre-length space and let $\sim$ be an equivalence relation on $X$. The quotient time separation $\ttau: X \times X \to [0,\infty]$ is defined as
\begin{equation}
\ttau([x],[y]):=\sup \{ \sum_{i=1}^n \tau(x_i,y_i) \mid x \sim x_1 \leq y_1 \sim x_2 \leq y_2 \sim \ldots \sim x_n \leq y_n \sim y, n \in \N \}.
\end{equation}
We call a sequence $(x_1,y_1,\ldots, x_n,y_n)$ as above an $n$-chain from $[x]$ to $[y]$, or simply a chain.
We define $[x] \tll [y]: \iff \ttau([x],[y]) > 0$ and $[x] \tleq [y]: \iff \{ \sum_{i=1}^n \tau(x_i,y_i) \mid x \sim x_1 \leq y_1 \sim x_2 \leq y_2 \sim \ldots \sim x_n \leq y_n \sim y, n \in \N \} \neq \emptyset$. Moreover, recall the definition of the quotient semi-metric, 
\begin{equation}
\td([x],[y]):= \inf \{\sum_{i=1}^n d(x_i,y_i) \mid x \sim x_1, x_{i+1} \sim y_i, y_n \sim y, n \in \N \}, 
\end{equation}
cf.\ \cite[Definition 3.1.12]{BBI}.
\end{defin}

Under the assumption that the identified subsets are closed and non-timelike locally isolating, the Lorentzian amalgamation is always a Lorentzian pre-length space.
\begin{defin}[Lorentzian amalgamation]
\label{def: gluing}
Let $(X_1,d_1,\ll_1,\leq_1,\tau_1)$ and $(X_2,d_2,\ll_2,\leq_2,\tau_2)$ be two Lorentzian pre-length spaces. 
Let $A_1$ and $A_2$ be closed and non-timelike locally isolating subspaces of $X_1$ and $X_2$, respectively. 
Let $f : A_1 \rar A_2$ be a locally bi-Lipschitz homeomorphism\footnote{This means that every $a \in A_1$ has a neighbourhood $U \subseteq A_1$ such that $f|_U:U \to f(U)$ and its inverse are Lipschitz.} and suppose that the causality of $A_1$ and $A_2$ are compatible in the following sense: for all $a \in A_1$ we have $I_1^{\pm}(a) \neq \emptyset \iff I_2^{\pm}(f(a)) \neq \emptyset$. 
Let $(X_1 \sqcup X_2, d, \ll, \leq, \tau)$ be the Lorentzian disjoint union of $X_1$ and $X_2$ and consider the equivalence relation $\sim$ on $X_1 \sqcup X_2$ generated by $a \sim f(a)$ for all $a \in A_1$. Then $((X_1 \sqcup X_2)/\sim, \td, \tll, \tleq, \ttau)$ is called the Lorentzian amalgamation of $X_1$ and $X_2$ and is denoted by $X_1 \sqcup_A X_2$.
\end{defin}

Note that Lorentzian amalgamation can be defined with surprisingly little compatibility of the structure in the two different subsets. However, without such assumptions this construction is usually poorly behaved, see \cite[Example 3.3.4]{BR22} for an extreme counterexample.
In particular, there is no hope for any compatibility results.
Thus, unless explicitly stated otherwise, we will always assume that $f$ is $\tau$-preserving and $\leq$-preserving, i.e., $\tau(a,b)=\tau(f(a),f(b))$ and $a \leq b \iff f(a) \leq f(b))$ for all $a,b \in A_1$. In similar fashion, when dealing with Lorentzian amalgamation, $X$ will always denote the amalgamation of two Lorentzian pre-length spaces $X_1$ and $X_2$ along the subsets $A_1$ and $A_2$, and we will identify $X_i$ with its image $\pi(X_i)$. \\

We also want to explain and justify the notation we will be using from now on. Dealing with causality in three different spaces at once can be overwhelming. 
Therefore, we decided to highlight in the notation of causal and timelike pasts and futures in which space the respective set is considered. 
Moreover, we prefer to not use the identifying map $f$ when denoting points in $A_2$. 
Rather, we feel it is advantageous to highlight with an index from which space a point comes originally, in addition to using bracket-notation for equivalence classes for points in $X$. For example, if $[x] \in X_1 \setminus A_1$, then $[x] = \{x^1\}$. If $[a] \in A$, then we write $[a]=\{a^1,a^2\}$ instead of $\{a, f(a)\}$. 
Further, the (timelike) future of $x^1 \in X_1$ is denoted by $I_1^+(x^1)$, while the future of its equivalence class in $X$ is denoted by $I_X^+([x]$). Similar notation will be used for sets related to the causal relation, e.g., $J_2(x^2,y^2)$ and $J_X([x],[y])$ for causal diamonds in $X_2$ and $X$, respectively.
Finally, the usage of $A$ is technically a bit misleading since this set is not yet introduced. On the one hand, we feel it is quite clear what is meant by $A$, namely the shared set in the glued space. On the other hand, one can set $A:=\pi(A_1 \sqcup A_2)$ for a rigorous definition. 
Compared to \cite{BR22}, one tiny detail in notation will appear different to an attentive reader, however: it is essentially never necessary to mention the original relations $\ll_1, \leq_1, \ll_2, \leq_2$ or the original time separation functions $\tau_1, \tau_2$. Rather, it is enough to deal with the corresponding notions in the disjoint union. 
That is, in the language of Definition \ref{def: gluing} we have $x^1 \leq y^1$ if and only if $x^1 \leq_1 y^1$, and thus omitting the index in $\leq$ leads to easier readability. \\

For a more detailed investigation of Lorentzian amalgamation and its elementary properties, we refer to \cite{BR22}. We will nevertheless collect some of the most useful facts for the current work here. We emphasize one final time that we only consider gluing constructions where the identifying map is well-behaved (otherwise, the following properties are of course not valid).

\begin{prop}[Useful facts about Lorentzian amalgamation]
\label{prop: gluing properties}
Let $X:=\amA$ be the Lorentzian amalgamation of two Lorentzian pre-length spaces $X_1$ and $X_2$.
\begin{itemize}
\item[(i)] 
$\ttau$ has the following simplified form:
\begin{equation}
\ttau([x],[y])=
\begin{cases} \tau(x^i,y^i) & x^i,y^i \in X_i, i \in \{1,2\}, \\
\underset{[a] \in J_X([x],[y]) \cap A}{\sup} \{ \tau(x^i,a^i)+\tau(a^j,y^j) \} & x^i \in X_i, y^j \in X_j, \{i,j\} = \{1,2\}.
\end{cases}
\end{equation}
Moreover, if $\ttau([x],[y]) > 0$, the supremum can be approximated by only considering timelike chains.
\item[(ii)] $\ttau([x],[y] \geq \tau(x^i,y^j), \td([x],[y]) \leq d(x^i,y^j), x^i \ll y^j \Rightarrow [x] \tll [y]$ and $x^i \leq y^j \Rightarrow [x] \tleq [y], i,j \in \{1,2\}$.
\end{itemize}
\end{prop}
\begin{proof}
See \cite[Remark 3.1.3, Proposition 3.3.7, Lemma 3.3.3, Corollary 3.1.7]{BR22}.
\end{proof}
\begin{rem}[Convergence in amalgamtion]
\label{rem: convergence in amalg}
Since $\td \leq d$, for any sequence $p_n^i \to p^i$ it follows that $[p_n] \to [p]$. Conversely, if $[p_n] \to [p]$ and $[p] \in X_i$, then for all subsequences in $X_i$ we have $p_{n_k}^i \to p^i$. In particular, there always exists a converging subsequence in at least one of the original spaces, 
cf. \cite[Remark 3.3.5]{BR22}. 
Indeed, if $[p] \in X_i \setminus A_i$, i.e., $[p] = \{p^i\}$, then $[p_n]=\{p_n^i\}$ for large enough $n$ as well. 
Let $U$ be an open neighbourhood of $p^i$ in $X_i$ that does not meet $A_i$ (in fact, $p^i$ has a neighbourhood basis of such sets). 
Then $\pi(U)$ is a neighbourhood of $[p]$ since $\pi^{-1}(\pi(U))=U$ is open, and as such contains all but finitely many sequence members $[p_n]$. 
Thus, all but finitely many members $p_n^i$ are contained in $U$ and so $p_n^i \to p^i$ follows.
Let $[p] \in A$.
For any $[p_n] \in X$, there exists $i \in \{1,2\}$ such that $p_n^i \in X_i$, thus there exists a subsequence in at least one of the original spaces. 
Say $(p_{n_k}^1)_{k \in \N}$ is a subsequence in $X_1$. Let $U_1$ be an open neighbourhood of $p^1$ in $X_1$. 
Then we find an open neighbourhood $U_2$ of $p^2$ in $X_2$ such that $f(U_1 \cap A_1)=U_2 \cap A_2$. Then $U:=\pi(U_1 \sqcup U_2)$ is an open neighbourhood of $[p]$ in $X$ and as such $[p_{n_k}] \in U$, i.e., $p^1_{n_k} \in U_1$ for almost all $k$. Hence $p^1_{n_k} \to p^1$.
\end{rem}
This remark remains valid for gluing finitely many spaces. Moreover, all statements in this work concerning gluing easily generalize to finitely many spaces, it is just more convenient to formulate everything with only two spaces at hand. \\

At last, we will briefly present the causal ladder for Lorentzian pre-length spaces. As noted in the introduction, this was first introduced in \cite{KS18} and further developed in \cite{ACS20}. Both works formulated the steps for Lorentzian length spaces, but in the spirit of generality, we try to do everything in the setting of Lorentzian pre-length spaces and rather specify the precise additional assumptions that are needed. \\

On a somewhat related note, there is another notion of local causal closure introduced in \cite[Definition 2.19]{ACS20} and also used in \cite{BGH21} which is more general than Definition \ref{def: basic prop}(iii) and appears more advantageous in intrinsic and/or causally path-connected settings. We stick to the original definition for two reasons: on the one hand, we prefer to not use intrinsic notions explicitly whenever possible, and on the other hand we use local causal closedness anyways only when we also assume strong causality, which is why we do not really lose any generality.
\begin{defin}[Preparing definitions for causal ladder]
Let $X$ be a Lorentzian pre-length space.
\begin{itemize}
\item[(i)] $X$ is distinguishing if $I^{\pm}(x)=I^{\pm}(y) \Rightarrow x=y$ for all $x,y \in X$.
\item[(ii)] $X$ is reflective if $I^{\pm}(x) \subseteq I^{\pm}(y) \Rightarrow I^{\mp}(y) \subseteq I^{\mp}$ for all $x,y \in X$.
\item[(iii)] The relation $K$ is defined as the smallest transitive and closed relation that contains $\leq$.
\end{itemize}
\end{defin}

\begin{defin}[The causal ladder]
Let $X$ be a Lorentzian pre-length space.
\begin{itemize}
\item[(i)] $X$ is chronological if $\ll$ is irreflexive, i.e., $x \not\ll x$ for all $x \in X$.
\item[(ii)] $X$ is causal if $\leq$ is antisymmetric, i.e., $x \leq y$ and $y \leq x$ imply $x=y$.
\item[(iii)] $X$ is non-totally imprisoning if for every compact set $B \subseteq X$ there exists a constant $C>0$ such that $L^d(\gamma) \leq C$ for all causal curves $\gamma$ contained in $B$.
\item[(iv)] $X$ is strongly causal if $\mathcal{I}:=\{I(x,y) \mid x,y \in X\}$ forms a subbasis for the (metric) topology $\mathcal{D}$ on $X$.
\item[(v)] $X$ is $K$-causal\footnote{In \cite{ACS20}, this is called stable causality, which is equivalent in the smooth case.} if the relation $K$ is antisymmetric.
\item[(vi)] $X$ is causally continuous if it is distinguishing and reflective.
\item[(vii)] $X$ is causally simple if it is causal and $J^{\pm}(x)$ is closed for all $x \in X$.
\item[(viii)] $X$ is globally hyperbolic if it is non-totally imprisoning and $J(x,y)$ is compact for all $x,y \in X$.
\end{itemize}
\end{defin}

\begin{thm}[Implications of the causal ladder]
Let $X$ be a Lorentzian length space.
With the additional assumption of local compactness in $(v) \Rightarrow (iv)$, each step of the causal ladder implies the previous one.
\end{thm}
\begin{proof}
These implications have been established in \cite[Theorem 3.26]{KS18} and \cite[Theorem 3.16]{ACS20}.
\end{proof}
\section{Basics and preparations}
\begin{prop}[Past and future representation]
\label{pastfuturerep}
Let $X$ be a Lorentzian pre-length space.
\begin{itemize}
\item[(i)] $J^{\pm}(x)=\cup_{y \in J^{\pm}(x)}J^{\pm}(y)$ for all $x \in X$.
\item[(ii)] If $X$ is interpolative, then $I^{\pm}(x)=\cup_{y \in I^{\pm}(x)}I^{\pm}(y)$ for all $x \in X$.
\end{itemize}
\end{prop}
\begin{proof}
(i) ``$\subseteq$'' holds anyways since $x \in J^{\pm}(x)$. Conversely, if $z \in J^{\pm}(y)$ for some $y \in J^{\pm}(x)$, then $z \in J^{\pm}(x)$ by the transitivity of $\leq$. \\

(ii) ``$\supseteq$'' follows as above via the transitivity of $\ll$. Conversely, if, say, $z \in I^+(x)$, then since $X$ is interpolative we find $y \in X$ such that $x \ll y \ll z$, hence $z \in I^+(y)$ and $y \in I^+(x)$.
\end{proof}
\begin{ex}[Being interpolative is a necessary condition]
Without requiring that $X$ is interpolative, the statement in Proposition \ref{pastfuturerep}(ii) is wrong in general. Consider, e.g., Minkowski space with an open rectangle removed, depicted as in Figure \ref{fig: interpol}.
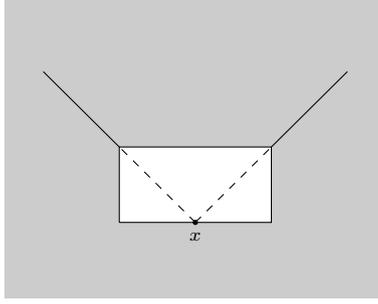
\begin{figure}
\begin{center}
\begin{tikzpicture}
\fill[black!20] (-1.5,-1) rectangle (3.5,3);
\draw [fill=white] (0,0) -- (2,0) -- (2,1) -- (0,1) -- (0,0);
\draw[dashed] (1,0) -- (2,1);
\draw[dashed] (1,0) -- (0,1);
\draw (2,1) -- (3,2);
\draw (0,1) -- (-1,2);
\begin{scriptsize}
\coordinate [circle, fill=black, inner sep=0.7pt, label=270: {$x$}] (A1) at (1,0);


\end{scriptsize}
\end{tikzpicture}
\end{center}
\caption{In this case, $I^{+}(x)=\cup_{y \in I^{+}(x)}I^{+}(y)$ does not hold.}
\label{fig: interpol}
\end{figure}
Then for a point $x$ on the lower boundary of the rectangle equality does not hold in the future case, since points on the upper boundary are not captured in any future of points in the future of $x$.
\end{ex}
\begin{lem}[Quotient relation reformulation]
\label{relation reformulation}
In an amalgamation $X:= \amA$, we have $[x] \tleq [y]$ if and only if $x^i \leq y^i$ for $i \in \{1,2\}$ or there exists $[a] \in A$ such that $x^i \leq a^i \sim a^j \leq y^j$ for $\{i,j\} = \{1,2\}$. Similarly for $\tll$.
\end{lem}
\begin{proof}
``$\Leftarrow$'' of the statement follows immediately by Proposition \ref{prop: gluing properties}(ii).
Conversely, $[x] \tleq [y]$ states by definition that we find a causal chain of the form\footnote{This follows by \cite[Remark 3.1.7]{BR22}: points outside of $A$ form their own equivalence class, so these ``trivial'' equivalences can be immediately cut out via the transitivity of the relations.} $x^i \leq a_1^i \sim a_1^j \leq a_2^j \sim a_2^i \leq \ldots \sim a_n^j \leq y^j$ (whether the chain starts and ends in the same space does not matter for the proof). By the causal preservation of the identifying map $f: A_1 \to A_2$, we have $a_k^1 \leq a_{k+1}^1 \iff a_k^2 \leq a_{k+1}^2$. Thus, one can ``shorten'' $a_k^i \leq a_{k+1}^i \sim a_{k+1}^j \leq a_{k+2}^j$ to $a_k^i \leq a_{k+2}^i$.
If the two endpoints in the chain are from the same space, then the chain reduces to a relation in the original space. If the two endpoints are from different spaces, then one single point in $A$ is needed to cross to the other side. \\

With Proposition \ref{prop: gluing properties}(i), the timelike case follows in complete analogy. 
\end{proof}
\begin{prop}[Glued past and future representation]
\label{causal rep}
Let $X:=\amA$ be the amalgamation of two Lorentzian pre-length spaces $X_1$ and $X_2$. Then we have the following decomposition of (pasts and) futures in $X$:
\begin{itemize}
\item[(i)] If $[x] \in A$, then $I_X^{\pm}([x])=\pi(I^{\pm}_1(x^1) \sqcup I^{\pm}_2(x^2))$.
\item[(ii)] If, say, $[x]=\{x^1\}$, then $I_X^{+}([x])=\pi(I_1(x^1) \sqcup (\cup_{x^1 \ll a^1} I_2^+(a^2)))\footnote{By $\cup_{x^1 \ll a^1} I_2^+(a^2)$ we mean the union of $I_2^+(a^2)$ over all $a^1 \in I_1^+(x^1) \cap A_1$. This should be clear as the letter $a$ is usually reserved for points from the identified sets and the expression anyways does not make sense for points outside of $A$.}$, and similarly for the past case.
\end{itemize}
The same holds for causal pasts and futures.
\end{prop}
\begin{proof}
(i) This is essentially one half of \cite[Lemma 3.3.8]{BR22}: ``$\supseteq$'' follows immediately by Proposition \ref{prop: gluing properties}(ii). 
Conversely, suppose $[y] \in I_X^+([x])$. Since $[x]=\{x^1,x^2\}$ we can find a chain as in Lemma \ref{relation reformulation} that starts in $x^i$ and ends in $y^i$, which then reduces to an original relation. \\

(ii) Here, ``$\supseteq$'' follows immediately as well.
Let $[y] \in I^+_X([x])$. Then by Lemma \ref{relation reformulation}, either $x^1 \ll y^1$ or $x^1 \ll a^1 \sim a^2 \ll y^2$.
These conditions precisely describe the right hand side of the equation.
\end{proof}

At present, it seems that strong causality is the ``correct threshold'' at which Lorentzian pre-length spaces behave reasonably (in the same sense as global hyperbolicity being the right notion for physically plausible spacetimes). This is because at the stage of strong causality, the topology and causality genuinely interact with each other in a desirable way.
In some cases, it is very convenient that $\mI$ is not just a subbasis but a basis for the topology. As it turns out, 
the deciding property is that of non-timelike local isolation. 
This result was first covered in \cite[Theorem 1.6.33]{Ber20}.
\begin{lem}[Basis for Alexandrov topology]
\label{lem: nlti basis}
Let $X$ be a strongly causal and non-timelike locally isolating Lorentzian pre-length space. Then $\mI$ forms a basis for the topology.
\end{lem}
\begin{proof}
Let $U$ be a $\mathcal{D}$-neighbourhood of a point $p \in X$. 
Then by strong causality of $X$ we find points $x_1,y_1, \ldots, x_n,y_n$ such that $p \in \cap_{i=1}^n I(x_i,y_i) \subseteq U$. 
Clearly, this intersection is an open causally convex neighbourhood of $p$. Hence, by the non-timelike local isolation of $X$ we find points $q_-,q_+ \in \cap_{i=1}^n I(x_i,y_i)$ such that $q_- \ll p \ll q_+$. 
Then $p \in I(q_-,q_+) \subseteq J(q_-,q_+) \subseteq \cap_{i=1}^n I(x_i,y_i) \subseteq U$, showing that any neighbourhood of any point contains a timelike diamond which contains that point. Indeed, it is even possible to choose the ``endpoints'' of the diamond to lie in the original neighbourhood as well.
\end{proof}

\begin{cor}[Strongly causal Lorentzian length spaces]
\label{llsbasis}
Let $X$ be a strongly causal Lorentzian length space. Then $\mI$ forms a basis for the topology.
\end{cor}
\begin{proof}
This follows immediately by the definition of localizing neighbourhoods. Indeed, let $p \in V:= \cap_{i=1}^n I(x_i,y_i) \subseteq U$ be as above, where $U$ is a localizable neighbourhood of $p$. Since $I^{\pm}(p) \cap U \neq \emptyset$ and $X$ is causally path-connected, we find a timelike curve $\gamma:[a,b] \to X$ through $p$, which by continuity is contained for some time in $V$, i.e., $p \in \gamma((-\varepsilon,\varepsilon)) \subseteq V$. Since $V$ is causally convex, any timelike diamond formed by points on $\gamma((-\varepsilon,\varepsilon))$ is contained in $V$ and hence in $U$.
\end{proof}

In particular, causal path-connectedness and the absence of $\ll$-isolated points imply that a Lorentzian pre-length space $X$ is non-timelike locally isolating.
The following three statements are slight generalizations of some of the points in \cite[Theorem 3.26]{KS18}.

\begin{lem}[Reformulating strong causality]
\label{str caus2}
In a causally path connected Lorentzian pre-length $X$ space with no $\ll$-isolated points, strong causality is equivalent to the non-existence of almost closed causal curves.
\end{lem}

\begin{proof}
Strong causality implying the formulation with curves is already shown in \cite[Lemma 2.38]{KS18}.
Conversely, let $x \in X$ and let $U$ be a neighbourhood of $x$. 
Then we find a neighbourhood $V$ of $x$ contained in $U$ such that $J(p,q) \subseteq U$ for all $p,q \in V$.
Since $I^{\pm}(x) \neq \emptyset$, there is a timelike curve $\gamma : [a,b] \to X$ through $x$. 
In particular, setting $\gamma(0)=:x$, there is $\varepsilon >0$ such that $\gamma([-\varepsilon, \varepsilon]) \subseteq V$. Then by assumption $J(p,q) \subseteq U$ for $p,q \in \gamma([-\varepsilon,\varepsilon])$. Since $I(p,q) \subseteq J(p,q)$, we have found an element of the subbasis of the Alexandrov topology inside an arbitrary neighbourhood of $x$.
\end{proof}

\begin{lem}[Global hyperbolicity implies strong causality]
\label{causal ladder0}
A causally path-connected, locally causally closed and globally hyperbolic Lorentzian pre-length space $X$ with no $\ll$-isolated points is strongly causal. In particular, such a space is locally compact.
\end{lem}

\begin{proof}
If $X$ were not strongly causal, by Lemma \ref{str caus2} there exists $x \in X$ and a neighbourhood $U$ of $x$ such that for all neighbourhoods $V \subseteq U$ of $x$ there are points $p,q \in V$ and a causal curve between them leaving $U$. 
Since $I^{\pm}(x) \neq \emptyset$ and $X$ is causally path-connected, there is a timelike curve $\gamma: [a,b] \to X$ through $x$. 
By choosing points on this curve close enough to $x$, we find $p,q \in \gamma([a,b]) \cap U, p \ll x \ll q$, where $U$ is the supposed neighbourhood where strong causality fails. 
In particular, $x \in I(p,q)$. Since this set is open, there exists $r>0$ such that $B_r(x) \subseteq I(p,q) \subseteq J(p,q)$, where $J(p,q)$ is compact by assumption. 
Since $X$ is supposed to be not strongly causal, for each $r' <r$ there is a causal curve with endpoints in $B_{r'}(x)$ that leaves $U$. Consider a sequence $(\gamma_n)_{n \in \N}$ of such curves with endpoints $p_n,q_n \in B_{\frac{1}{n}}(x)$ for large enough $n$. 
The image of each $\gamma_n$ is contained in the compact and causally convex set $J(p,q)$, and since $X$ is non-totally imprisoning, there exists a constant $C>0$ such that $L_d(\gamma_n) \leq C$ for all $n$. 
Parameterized by $d$-arclength, each curve $\gamma_n$ has $[0,l_n]$ as domain, with $l_n \leq C$ and Lipschitz constant 1. 
Denote by $\tilde{\gamma}: [0,C] \to X$ the reparameterization of $\gamma_n$ given by $\tilde{\gamma_n}(t):=\gamma_n(\frac{l_n}{C}t)$. 
Then 
\begin{equation}
d(\tilde{\gamma_n}(s),\tilde{\gamma_n}(t))=d(\gamma_n \big( \frac{l_n}{C}s \big) , \gamma_n \big( \frac{l_n}{C}t \big) ) =\frac{l_n}{C}|t-s| \leq |t-s|.
\end{equation}
In summary, $\tilde{\gamma_n} : [0,C] \to X$ and Lip$(\tilde{\gamma_n}) \leq 1$ for all $n$. Thus, all requirements of the Limit curve theorem \cite[Theorem 3.7]{KS18} are satisfied.
Since $p_n,q_n \to x$, we infer the existence of a causal loop through $x$. 
Thus, $X$ is not causal, a contradiction to $X$ being non-totally imprisoning, cf.\ \cite[Theorem 3.26(ii)]{KS18} (the cited theorem is of course valid in a causally path-connected Lorentzian pre-length space). \\

Concerning local compactness, let $p \in X$ and let $U \subseteq X$ be a neighbourhood of $p$. By strong causality we find points $x_1,y_1,\ldots, x_n,y_n$ such that $p \in \cap_{i=1}^n I(x_i,y_i) \subseteq U$. Thus, $\cap_{i=1}^n J(x_i,y_i)$ is a compact neighbourhood of $p$ as each $J(x_i,y_i)$ is compact.
\end{proof}

\begin{lem}[Strong causality implies non-total imprisonment]
\label{causal ladder1}
A strongly causal, locally causally closed and $d$-compatible Lorentzian pre-length space is non-totally imprisoning.
\end{lem}
\begin{proof}
This is just extracting the right properties of a Lorentzian length space needed in \cite[Lemma 3.12, Lemma 3.15 \& Theorem 3.26]{KS18}.
\end{proof}

\begin{cor}[Conditions for global causal closedness]
A globally hyperbolic, causally path-connected and locally causally closed Lorentzian pre-length space $X$ with no $\ll$-isolated points is globally causally closed, i.e., for any $p_n \to p, q_n \to q$ with $p_n \leq q_n$ for all $n$ we have $p \leq q$.
\end{cor}
\begin{proof}
This is a consequence of the Limit curve theorem \cite[Theorem 3.7]{KS18}: let $p_n \to p, q_n \to q$ with $p_n \leq q_n$ for all $n$. 
If $p = q$ we are done so assume $p \neq q$.
By Lemma \ref{causal ladder0}, $X$ is strongly causal, so there exist timelike diamonds $I(x_1,y_1)$ and $I(x_2,y_2)$ containing $p$ and $q$ as well as infinitely many sequence members $p_n$ and $q_n$, respectively. 
Then $p_n,q_n \in I(x_1,y_2) \subseteq J(x_1,y_2)$ by push-up, cf.\ \cite[Lemma 2.10]{KS18}. Moreover, $J(x_1,y_2)$ is compact by assumption and hence closed, so $p,q \in J(x_1,y_2)$ as well. By the causal path-connectedness we infer the existence of causal curves $\gamma_n$ connecting $p_n$ and $q_n$, which are all contained in $J(x_1,y_2)$. 
By the non-total imprisonment of $X$, we find $C>0$ such that $L^d(\gamma_n) \leq C$ for all $n$. 
With a reparameterization argument as in Lemma \ref{causal ladder0}, we can apply the Limit curve theorem to obtain a causal curve from $p$ to $q$, so in particular $p \leq q$.
\end{proof}

\begin{lem}[Recovering subsets]
\label{recovering subsets}
Let $X_1$ and $X_2$ be topological spaces, $A_i \subseteq X_i$, and $f:A_1 \to A_2$ bijective. Consider the quotient space of $X_1 \sqcup X_2$ generated by the equivalence relation $a \sim f(a)$ for all $a \in A_1$.
Let $Y_i \subseteq X_i$ be subsets such that $\pi(Y_1 \cap A_1)=\pi(Y_2 \cap A_2)$. Then $\pi^{-1}(\pi(Y_1 \sqcup Y_2))=Y_1 \sqcup Y_2$.
\end{lem}
\begin{proof}
We first show that $\pi^{-1}(\pi((Y_1 \cap A_1) \sqcup (Y_2 \cap A_2)))=(Y_1 \cap A_2) \sqcup (Y_2 \cap A_2)$. 
The inclusion ``$\supseteq$'' is always true.
So let $x \in \pi^{-1}(\pi((Y_1 \cap A_1) \sqcup (Y_2 \cap A_2)))$, then $\pi(x)=[x] \in \pi((Y_1 \cap A_1) \sqcup (Y_2 \cap A_2))$.
Clearly, $[x]=\{x^1,x^2\}$, so there exists $i \in \{1,2\}$ such that $x^i \in Y_i \cap A_i$, say without loss of generality $i=1$. Since $\pi(Y_1 \cap A_1) = \pi(Y_2 \cap A_2)$ by assumption, it then follows that $[x] \in \pi(Y_2 \cap A_2)$ as well. So we have $x^1,x^2 \in Y_1 \cap A_2 \sqcup Y_2 \cap A_2$. \\

The full statement easily follows from the following calculation, since, roughly speaking, for points outside of $A$ there is a unique preimage:
\begin{align*}
\pi^{-1}(\pi(Y_1 \sqcup Y_2)) & = \pi^{-1}(\pi((Y_1 \cap A_1) \cup (Y_1 \setminus A_1) \sqcup (Y_2 \cap A_2) \cup (Y_2 \setminus A_2))) \\
& = \pi^{-1}(\pi((Y_1 \cap A_1 \sqcup Y_2 \cap A_2) \cup (Y_1 \setminus A_1 \sqcup Y_2 \setminus A_2))) \\
& = \pi^{-1}(\pi(Y_1 \cap A_1 \sqcup Y_2 \cap A_2) \cup \pi(Y_1 \setminus A_1 \sqcup Y_2 \setminus A_2)) \\
& = \pi^{-1}(\pi(Y_1 \cap A_1 \sqcup Y_2 \cap A_2)) \cup \pi^{-1}(\pi(Y_1 \setminus A_1 \sqcup Y_2 \setminus A_2)) \\
& = (Y_1 \cap A_1 \sqcup Y_2 \cap A_2) \cup (Y_1 \setminus A_1 \sqcup Y_2 \setminus A_2) = Y_1 \sqcup Y_2.
\end{align*}
\end{proof}

\section{Gluing Lorentzian length spaces}
Next, we investigate the question of whether the amalgamation of Lorentzian length spaces is again a Lorentzian length space. We cover each defining property separately and obtain the final result in the end.

\begin{ex}[Being intrinsic vs.\ causal path-connectedness]
Note that an intrinsic Lorentzian pre-length space is not necessarily causally path-connected, as there might be points which are null related but do not have a null curve joining them. Indeed, consider the subspace (in the sense of Definition \ref{def: subspace}) of the Minkowski plane $J_+((0,0)) \setminus \{(x,y) \mid 1 < x=y < 2\}$, i.e., a causal future with a segment of a light ray removed. Then $(1,1) \leq (2,2), \tau((1,1),(2,2))=0$, but there is no causal curve connecting the two points, 
see Figure \ref{fig: intrinsic} on the left. 
Moreover, the two properties are independent, as the open subset of the Minkowski plane viewed as a subspace depicted on the right in Figure \ref{fig: intrinsic} is causally path-connected but not intrinsic.
\begin{figure}
\begin{center}
\begin{tikzpicture}
\draw [dashed,fill=black!20] (-3,3)--(0,0)--(3,3);
\draw [thick] (-3,3) -- (0,0) -- (1,1);
\draw [thick] (2,2) -- (3,3);
\begin{scriptsize}
\coordinate [circle, fill=black, inner sep=0.7pt] (A1) at (1,1);

\coordinate [circle, fill=black, inner sep=0.7pt] (A1) at (2,2);
\coordinate [circle, fill=black, inner sep=0.7pt, label=270: $0$] (A1) at (0,0);


\end{scriptsize}

\draw [dotted,fill=black!20] (5,0)--(6.5,1.5)--(5,3)--(7,3)--(8.5,1.5)--(7,0)--(5,0);
\coordinate [circle, fill=black, inner sep=0.7pt] (p) at (6.25,0.5);

\coordinate [circle, fill=black, inner sep=0.7pt] (q) at (6.25,2.5);
\draw (p) .. controls (6.75,1.4) and (6.75,1.8) ..(q);

\draw[dashed] (6.25,1.25)--(6.25,1.75);
\draw (p) -- (6.25,1.25);
\draw (q) -- (6.25,1.75);

\end{tikzpicture}
\end{center}
\caption{The space on the left is intrinsic but not causally path-connected. The space on the right is causally path-connected but not intrinsic.}
\label{fig: intrinsic}
\end{figure}
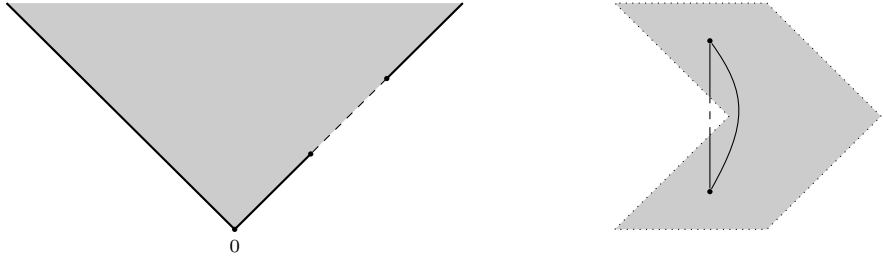
\end{ex}

\begin{prop}[Causal path-connectedness]
\label{prop: caus pc}
Let $X_1$ and $X_2$ be two causally path-connected Lorentzian pre-length spaces. Then $X:=X_1 \sqcup_A X_2$ is causally path-connected.
\end{prop}
\begin{proof}
Let $[x],[y] \in X$ with $[x] \tleq [y]$. If $[x],[y] \in X_i$, then $x^i \leq y^i$ and the projection of a causal curve in $X_i$ joining $x^i$ and $y^i$ is a causal curve in $X$ joining $[x]$ and $[y]$. 
Thus, the only case left to check is, up to symmetry, $[x]=\{x^1\}$ and $[y]=\{y^2\}$. 
In this case, $[x] \tleq [y]$ implies that there exists $[a] \in A$ such that $x^1 \leq a^1 \sim a^2 \leq y^2$, cf.\ Lemma \ref{relation reformulation}. Since $X_1$ and $X_2$ are causally path-connected, there exist causal curves from $x^1$ to $a^1$ in $X_1$ and from $a^2$ to $y^2$ in $X_2$, respectively. The concatenation of their projections is a causal curve from $[x]$ to $[y]$ in $X$. \\

Using Proposition \ref{prop: gluing properties}(i), the timelike case is entirely analogous.
\end{proof}

Before dealing with the property of being intrinsic, we state the following lemma.

\begin{lem}[Quotient length of causal curves]
\label{lem: quot length}
Let $X:= \amA$ be the Lorentzian amalgamation of two Lorentzian pre-length spaces $X_1$ and $X_2$.
\begin{itemize}
\item[(i)] Let $\gamma:[a,b] \to X_1$ be a causal curve in $X_1$. Assume that $f$ is not necessarily $\tau$-preserving. Then $L_{\ttau}(\pi \circ \gamma) \geq L_{\tau}(\gamma)$.
If $f$ is $\tau$-preserving, then $L_{\ttau}(\pi \circ \gamma) = L_{\tau}(\gamma)$. The same is true for causal curves in $X_2$.

\item[(ii)] Let $\gamma:[a,c] \to X$ be a causal curve resulting from the concatenation of projections of two causal curves $\alpha:[a,b] \to X_1$ and $\beta:[b,c] \to X_2$, i.e., $\gamma=(\pi \circ \alpha) \ast (\pi \circ \beta)$.
Then $L_{\ttau}(\gamma) = L_{\tau}(\alpha) + L_{\tau}(\beta)$.
\end{itemize}
\end{lem}

\begin{proof}
(i) This follows immediately by Proposition \ref{prop: gluing properties}(ii). Let $a=t_0<t_1< \ldots <t_n=b$ be a partition of $[a,b]$.
Then $\sum_{i=1}^{n-1} \tau(\gamma(t_i)^1, \gamma(t_{i+1})^1) \leq \sum_{i=1}^{n-1} \ttau([\gamma(t_i)],[\gamma(t_{i+1})])$. Since the length of a causal curve is defined as the supremum over such expressions, the claim follows. The second claim in (i) follows with the same arguments and Proposition \ref{prop: gluing properties}(i), since in this case we have equality for each partition. \\

(ii) Note that we return here to the familiar setting of a structure preserving identification map $f$. This claim would be wrong in general without this assumption. Via (i) and \cite[Lemma 2.25]{KS18} we compute $L_{\ttau}(\gamma)=L_{\ttau}(\gamma|_{[a,b]}) + L_{\ttau}(\gamma|_{[b,c]}) = L_{\ttau}(\pi \circ \alpha) + L_{\ttau}(\pi \circ \beta) = L_{\tau}(\alpha) + L_{\tau}(\beta)$.
\end{proof}

\begin{prop}[Intrinsicality]
\label{prop: intrinsic}
Let $X_1$ and $X_2$ be two intrinsic Lorentzian pre-length spaces. Then $X:=X_1 \sqcup_A X_2$ is intrinsic.
\end{prop}
\begin{proof}
Let $[x],[y] \in X$. If $\ttau([x],[y])=0$, there is nothing to show, so suppose $\ttau([x],[y])>0$. 
If $[x],[y] \in X_i$, then $\ttau([x],[y])=\tau(x^i,y^i)$. As $X_i$ is intrinsic, we find a causal curve $\gamma$ from $x^i$ to $y^i$ such that $L_{\tau}(\gamma)>\tau(x^i,y^i) - \varepsilon$. 
Then $L_{\ttau}(\pi \circ \gamma) = L_{\tau}(\gamma)>\tau(x^i,y^i) - \varepsilon = \ttau([x],[y]) - \varepsilon$. In other words, if the two points are from the same space, the causality is unaffected by the gluing process and hence a well-approximating causal curve in the original space projects to a well-approximating causal curve in the quotient. \\

Again, the only case left to check is, up to symmetry, $[x]=\{x^1\}$ and $[y]=\{y^2\}$. 
By Proposition \ref{prop: gluing properties}(i), we have $\ttau([x],[y])=\sup\{\tau(x^1,a^1)+\tau(a^2,y^2) \mid [a] \in J_X([x],[y]) \cap A \}$. Let $[a]$ be such that $\tau(x^1,a^1)+\tau(a^2,y^2) > \ttau([x],[y]) - \varepsilon$. 
As $X_1$ and $X_2$ are intrinsic (and in fact, $X_1 \sqcup X_2$ is intrinsic as well), we find causal curves $\gamma_1$ from $x^1$ to $a^1$ in $X_1$ and $\gamma_2$ from $a^2$ to $y^2$ in $X_2$ such that, respectively, $L_{\tau}(\gamma_1) > \tau(x^1,a^1) - \varepsilon$ and $L_{\tau}(\gamma_2) > \tau(a^2,y^2) - \varepsilon$. Denote by $\gamma$ the concatenation of the projections of these curves.
Then by Lemma \ref{lem: quot length}(ii), $L_{\ttau}(\gamma)=L_{\tau}(\gamma_1)+L_{\tau}(\gamma_2) > \tau(x^1,a^1) - \varepsilon + \tau(x^2,a^2) - \varepsilon > \ttau([x],[y]) - 3 \varepsilon$. 
\end{proof}

\begin{prop}[Local causal closedness]
\label{lcc}
Let $X_1$ and $X_2$ be two locally causally closed, strongly causal and locally compact Lorentzian pre-length spaces. Then $X:=\amA$ is locally causally closed.
\end{prop}
\begin{proof}
At first, note that any subset of a causally closed neighbourhood is again causally closed.
Let $[x] \in X \setminus A$. Then there is an open neighbourhood $U_i \subseteq X_i$ of $x^i$ which does not meet $A_i$, and we can choose $U_i$ small enough to be contained in a causally closed neighbourhood of $x^i$ in $X_i$. Then $\pi(U_i)$ is a causally closed neighbourhood of $[x]$ in $X$.
Indeed, $\pi(U_i)$ is open as $\pi^{-1}(\pi(U_i))=U_i$ is open. 
Then for any $[p] \tleq [q]$ in $\pi(U_i)$ it follows that $p^i \leq q^i$. \\

Let $[x] \in A$. Let $V_1$ and $V_2$ be open and causally closed neighbourhoods of $x^1$ and $x^2$ in $X_1$ and $X_2$, respectively. Let $V_1'$ be such that $V_1' \cap A_1 = f^{-1}(V_2 \cap A_2)$ and let $V_2'$ be such that $V_2' \cap A_2=f(V_1 \cap A_1)$. Then $x^1 \in U_1 := V_1 \cap V_1'$ and $x^2 \in U_2 := V_2 \cap V_2'$ are open as well as causally closed and satisfy $f(U_1 \cap A_1)=U_2 \cap A_2$\footnote{We will use this construction of neighbourhoods which agree on $A$ via $f$ and share some property several times throughout this work. We will not give a detailed description in future cases.}. It follows that $U:=\pi(U_1 \sqcup U_2)$ is an open neighbourhood of $[x]$ in $X$, cf.\ Lemma \ref{recovering subsets}.
By strong causality we find points $p_1^1,q_1^1, \ldots, p_n^1,q_n^1$ and $a_1^2,b_1^2, \ldots, a_m^2,b_m^2$ such that $x^1 \in S_1 := \cap_{i=1}^n I_1(p_i^1,q_i^1) \subseteq U_1$ and $x^2 \in S_2 := \cap_{j=1}^mI_2(a_j^2,b_j^2) \subseteq U_2$. Then we find a neighbourhood $O_i$ of $x^i$ which is contained in $S_i$ such that $f(O_1 \cap A_1)=O_2 \cap A_2$. 
As $A_1$ and $A_2$ are non-timelike locally isolating, we find $b_-^i,b_+^i \in O_i \cap A_i$ such that $b_-^i \ll x^i \ll b_+^i, i=1,2$. 
In particular, $x^i \in I_i(b_-^i,b_+^i) \subseteq J_i(b_-^i,b_+^i) \subseteq S_i \subseteq U_i$ since $S_i$ is causally convex.
Additionally, by the local compactness of $X_i$, we can assume that both these timelike diamonds are contained in a compact neighbourhood as well.
We claim that (the closure of) $D:=I_X([b_-],[b_+])=\pi(I_1(b_-^1,b_+^1) \sqcup I_2(b_-^2,b_+^2))=:\pi(D_1 \sqcup D_2)$ is a causally closed neighbourhood of $[x]$ in $X$. \\

Let $([p_n])_{n \in \N}, [p_n])_{n \in \N}$ be two sequences in $D$ converging to $[p]$ and $[q]$, respectively, and suppose $[p_n] \tleq [q_n]$ for all $n \in \N$. 
This means that for all $n$ we either have $p_n^i \leq q_n^i$ or there exists $[a_n] \in A$ such that, say, $p_n^1 \leq a_n^1 \sim a_n^2 \leq q_n^2$.
Recall Remark \ref{rem: convergence in amalg} and suppose first that there exist subsequences such that $p_{n_k}^i \leq q_{n_k}^i$. 
Then $p^i \in [p], q^i \in [q]$ as $X_i$ is closed in $X$. 
Then $p^i \leq q^i$ follows since $\bar{D_i}$ is causally closed.
Otherwise, there exist subsequences that are up to symmetry of the form $p_{n_k}^1 \leq a_n^1 \sim a_{n_k}^2 \leq q_{n_k}^2$. 
By definition, $[a_n] \in J_X([p_n],[q_n]) \subseteq D$. 
In particular, $p_{n_k}^1, a_{n_k}^1 \in D_1$ as well as $a_{n_k}^2, q_{n_k}^2 \in D_2$ for all $k$.
As $D_i$ is relatively compact, we can extract from $(a_{n_k})_{k \in \N}$ a converging subsequence, say without loss of generality the whole sequence already converges to some $a^i \in \bar{D_i}$. Since $\bar{D_i}$ is causally closed, we infer $p^1 \leq a^1$ and $a^2 \leq q^2$. In particular, we have $[p] \tleq [a] \tleq [q]$ and hence $[p] \tleq [q]$, showing that $\bar{D}$ is a causally closed neighbourhood of $[x]$ in $X$.
\end{proof}
\begin{lem}[Localizing neighbourhoods and strong causality]
\label{lem: str caus loc nhood basis}
In a strongly causal, non-timelike locally isolating and localizable Lorentzian pre-length space $X$, each point has a neighbourhood basis of localizable timelike diamonds, hence $X$ is strongly localizable.
\end{lem}

\begin{proof}
Let $U$ be a localizing neighbourhood of $x \in X$.  
By strong causality and non-timelike local isolation we find, as in Lemma \ref{lem: nlti basis}, $p,q \in U$ such that $x \in I(p,q) \subseteq U$ (note that $X$ has no $\ll$-isolated points as it is localizable).
By the causal convexity of $I(p,q)$, every causal curve with endpoints in $I(p,q)$ is entirely contained in $I(p,q)$ and hence also in the original neighbourhood $U$. Thus, we can take the same d-compatibility constant $C>0$ for $I(p,q)$ that we used for $U$.
If $y \in I(p,q)$, then as $I(p,q)$ is a neighbourhood of $y$ and $X$ is non-timelike locally isolating, we find $y_-,y_+ \in I(p,q)$ such that $y_- \ll y \ll y_+$, which immediately yields $I^{\pm}(y) \cap I(p,q) \neq \emptyset$, so the non-triviality condition is satisfied. 
Finally, recall that $\omega_U(x,y)=\max\{L(\gamma) \mid \gamma \text{ is a causal curve from $x$ to $y$ in $U$} \}$. 
We claim that $\omega_U|_{I(p,q) \times I(p,q)} = \omega_{I(p,q)}$, i.e., the maximal curve between points in $I(p,q)$ that is contained in $U$ is already contained in $I(p,q)$. This follows immediately by the causal convexity of $I(p,q)$. \\

Timelike diamonds form a basis for the topology, cf.\ Lemma \ref{lem: nlti basis}, so $X$ is strongly localizable.
\end{proof}

Via Lemma \ref{lem: str caus loc nhood basis}, we obtain the following very slight generalization of \cite[Lemma 4.3]{GKS19}.

\begin{lem}[$\tau$ determines $\omega$]
\label{lem: tau and omega}
Let $X$ be a strongly causal, non-timelike locally isolating, intrinsic and localizable Lorentzian pre-length space. Then each point in $X$ has a neighbourhood basis of localizable neighbourhoods where $\omega$ agrees with $\tau$. In other words, $X$ is locally geodesic.
\end{lem}
\begin{proof}
Lemma \ref{lem: str caus loc nhood basis} establishes the existence of a causally convex and localizable neighbourhood basis. 
Being intrinsic is the only additional property of a strongly causal and localizable Lorentzian pre-length space that is required in the proof of \cite[Lemma 4.3]{GKS19}.
\end{proof}

\begin{prop}[Localizability]
\label{prop: loc}
Let $X_1$ and $X_2$ be two strongly causal, causally path-connected, locally compact, locally causally closed, intrinsic and localizable Lorentzian pre-length spaces. Then $X:= \amA$ is localizable.
\end{prop}
\begin{proof}
Let $[x] \in X \setminus A$. Then we find a timelike diamond $I_i(p^i,q^i)$ in $X_i$ which is a localizing neighbourhood of $x^i$ and does not meet $A_i$. Then $I_X([p],[q])=\pi(I_i(p^i,q^i))$ is a localizable neighbourhood of $[x]$. \\

Let $[x] \in A$. Similar to the construction in Proposition \ref{lcc}, we find a timelike diamond $I_X([p],[q])=\pi(I_1(p^1,q^1) \sqcup I_2(p^2,q^2))$ such that both original diamonds are localizing neighbourhoods in $X_1$ and $X_2$, respectively. We claim that $I_X([p],[q])$ is a localizing neighbourhood of $[x]$ in $X$.
Let $\gamma: [a,b] \to X$ be a causal curve in $I_X([p],[q])$. Then $\pi^{-1}(\gamma([a,b])$ is either a causal curve in $X_i$ or it consists of pieces of causal curves in $X_1$ and $X_2$. In the first case one can take the original constant $C_i >0$ for $\gamma$ by Proposition \ref{prop: gluing properties}(ii), so assume we are in the second case. 
Say $\gamma$ starts out in in $X_1$, leaves at $[y]=\gamma(s) \in A$ and enters again at $[z]=\gamma(t) \in A$. Then $\pi^{-1}(\gamma|_{[a,s]})$ is a causal curve in $X_1$ ending in $y^1$ and $y^1 \leq z^1$. 
As $X_1$ is causally path-connected, there is a causal curve between $y^1$ and $z^1$, any of which by definition is contained in $J_1(y^1,z^1) \subseteq I_1(p^1,q^1)$. Denote by $\alpha$ such a curve. 
Then the concatenation of $\pi^{-1}(\gamma|_{[a,s]})$ and $\alpha$ is a causal curve in $X_1$. 
In this way, we can piece together all parts of $\gamma$ that lie in $X_1$ to obtain one single causal curve, denoted by $\gamma_1$, which is still contained in $I_1(p^1,q^1)$. Doing the same procedure in $X_2$ and denoting the corresponding curve by $\gamma_2$, we infer that 
$L_{\td}(\gamma) \leq L_{\td}(\pi \circ \gamma_1) + L_{\td}(\pi \circ \gamma_2) \leq L_d(\gamma_1) + L_d(\gamma_2) \leq C_1 + C_2$. \\

Concerning the non-triviality condition, let $[y] \in I_X([p],[q])$.
Then $[p] \tll [y] \tll [q]$ and a timelike curve through $[y]$ must lie partially in this neighbourhood, hence $I_X^{\pm}([y]) \cap I_X([p],[q]) \neq \emptyset$. \\

Finally, we need to show the existence of (continuously varying) maximal causal curves in $I_X([p],[q])$. We essentially need to show that the map $\tilde{\omega}:I_X([p],[q]) \times I_X([p],[q]) \to [0,\infty]$, defined via
\begin{equation}
\tilde{\omega}([y],[z]):=\sup\{L_{\ttau}(\gamma) \mid \gamma \text{ is a causal curve from $[y]$ to $[z]$ contained in $I_X([p],[q])$}\},
\end{equation}
is continuous and always realizes the supremum. Note that $\ttau$ is intrinsic by Proposition \ref{prop: intrinsic}, and $\ttau = \tilde{\omega}$ on $I_X([p],[q])$ since the neigbourhood is causally convex.
Let $[y],[z] \in I_X([p],[q])$ with $[y] \tl [z]$. If $\ttau([y],[z])=0$, then either $\tau(y^i,z^i)=0$ or for all $[a] \in J_X([y],[z]) \cap A$ we have $\tau(y^1,a^1)=\tau(a^2,z^2)=0$. In both cases, we find a null curve between the points in $X$, which by definition is contained in $J_X([y],[z]) \subseteq I_X([p],[q])$. In the first case this is immediate with the projection and in the second case we concatenate two null curves which exist by the causal path-connectedness of $X_i$.
So we are left with the case $[y] \tll [z]$. 
If $[y], [z] \in X_i$, then by the $\tau$-preservation of $f:A_1 \to A_2$ and the fact that $\ttau$ is intrinsic, a maximal curve in $I_i(p^i,q^i)$ is still maximal in $I_X([p],[q])$. 
Indeed, $\tilde{\omega}([y],[z])=\ttau([y],[z])=\tau(y^i,z^i)=\omega_i(y^i,z^i)$, cf.\ Lemma \ref{lem: tau and omega}, where $\omega_i : I_i(p^i,q^i) \times I_i(p^i,q^i) \to [0,\infty)$ is the continuous map returning the maximal length of causal curves contained in $I_i(p^i,q^i)$.
So suppose without loss of generality $[y]=\{ y^1 \}, [z] = \{ z^2\}$.
Let $(\gamma_n)_{n \in \N}$ be a sequence of causal curves from $[y]$ to $[z]$ in $I_X([p],[q])$ such that their lengths converge to the supremum. 
By Proposition \ref{prop: gluing properties}(i), each of these causal curves can be assumed to consist of (the projections of) two maximal causal curves from $y^1$ to $a_n^1$ in $I_1(p^1,q^1)$ and from $a_n^2$ to $z^2$ in $I_2(p^2,q^2)$, respectively, for some $[a_n] \in A \cap J_X([y],[z]) \subseteq I_X([p],[q])$. 
This restriction does not decrease the length of curves and we are considering a sequence whose lengths converge to the supremum.
By the local compactness of $X_i$ we can assume that the neighbourhoods were chosen small enough such that even $J_i(p^i,q^i)$ is contained in a compact subset. 
Thus, we infer the existence of a converging subsequence of $(a_n^i)_{n \in \N}$, say without loss of generality $a_n^i \to a^i \in A_i \cap \widebar{J_i(p^i,q^i)}$.
By construction we have $x^1 \leq a_n^1 \sim a_n^2 \leq y^2$ for all $n$.
Since $J_i(p^i,q^i)$ can be assumed to be causally closed as well (just shrink everything), we infer $x^1 \leq a^1 \sim a^2 \leq y^2$ and hence $[a] \in A \cap J_X([y],[z]) \subseteq I_X([p],[q])$.
By hypothesis, we find maximal causal curves from $x^1$ to $a^1$ in $J_1(p^1,q^1)$ and from $a^2$ to $y^2$ in $J_2(p^2,q^2)$.
Denote the (projections of the) two pieces by $\gamma^1$ and $\gamma^2$, respectively, and their concatenation by $\gamma$. This is a maximal causal curve from $[y]$ to $[z]$ in $I_X([p],[q])$: using Lemma \ref{lem: quot length}, 
we clearly have $L_{\ttau}(\gamma_n)=L_{\tau}(\gamma_n^1) + L_{\tau}(\gamma_n^2)$, where $\gamma_n^1$ and $\gamma_n^2$ are the (maximal) parts of $\gamma_n$ in $I_1(p^1,q^1)$ and $I_2(p^2,q^2)$, respectively. Then we compute 
\begin{align*}
\tilde{\omega}([y],[z]) & = \lim_{n \to \infty} L_{\ttau}(\gamma_n) = 
\lim_{n \to \infty} (L_{\tau}(\gamma_n^1) + L_{\tau}(\gamma_n^2)) \\
& = \lim_{n \to \infty} L_{\tau}(\gamma_n^1) + \lim_{n \to \infty} L_{\tau} (\gamma_n^2) = 
\lim_{n \to \infty} \omega_1(y^1,a_n^1) + \lim_{n \to \infty} \omega(a_n^2,z^2) \\
& = \omega_1(y^1,a^1) + \omega_2(a^2,z^2) = 
L_{\tau}(\gamma^1) + L_{\tau}(\gamma^2) = 
L_{\ttau}(\gamma).
\end{align*}
Let $[y_m] \to [y]$ be a sequence, then $[y_m]=\{y_m^1\}$ for large $m$. This yields a sequence of jump points $[b_m]$ and corresponding maximal causal curves $\beta_m$ from $[y_m]$ to $[z]$. In particular, $([b_m])_{m \in \N}$ converges as well and so we obtain a causal curve $\beta$ from $[y]$ to $[z]$ through, say, $[b]$. By the above calculation we know that $\gamma$ is a maximal causal curve from $[y]$ to $[z]$, so we have $\lim_{m \to \infty} \tilde{\omega}([y_m],[z]) = L_{\ttau}(\beta) \leq L_{\ttau}(\gamma) = \tilde{\omega}([y],[z])$. 
Conversely, recall that $L_{\ttau}(\gamma)=L_{\tau}(\gamma_1) + L_{\tau}(\gamma_2)$. Say $\gamma_1$ is parameterized by $[0,1]$ and consider the map $t \mapsto L_{\tau}(\gamma_1|_{[t,1]})$. 
This map is continuous by \cite[Lemma 3.33]{KS18} (local continuity of $\tau$ is clearly sufficient in the proof, i.e., the curve is contained in a neighbourhood where $\tau$ is continuous). For appropriately small $\varepsilon>0$, let $t$ be such that $L_{\tau}(\gamma_1|_{[t,1]}) = L_{\tau}(\gamma_1) - \varepsilon$. 
As $y^1 \ll \gamma_1(t)^1$, we have $y_m^1 \ll \gamma_1(t)^1$ by the openness of $\ll$. 
Since $X_1$ is causally path-connected, we find a causal curve from $y_m^1$ to $\gamma_1(t)^1$, denote it by $\alpha_m$. Then we compute
\begin{align*}
\lim_{m \to \infty} \tilde{\omega}([y_m],[z]) & = \lim_{m \to \infty} L_{\ttau}(\beta_m) \geq 
\lim_{m \to \infty} L_{\ttau}((\pi \circ (\alpha_m \ast \gamma_1|_{[t,1]})) \ast (\pi \circ \gamma_2)) \\ 
& \geq L_{\ttau}(\gamma) - \varepsilon = 
\tilde{\omega}([y],[z]) - \varepsilon.
\end{align*}
This establishes that $\tilde{\omega}$ is continuous on $(X_1\setminus A_1 \times X_2 \setminus A_2) \cup (X_2 \setminus A_2 \times X_1 \setminus A_1)$. We also know that $\tilde{\omega}$ is continuous on $(X_1 \times X_1) \cup (X_2 \times X_2)$, which together cover $X \times X$. The only thing left to check is that the two cases fit together in the following scenario: let $[y] = \{y^1\}, [z] \in A$ and let $[z_k] \to [z]$ be a sequence such that $[z_k]=\{z_k^2\}$ for all $k$. We have to show $\tilde{\omega}([y],[z_k]) \to \tilde{\omega}([y],[z])$. 
By above calculations, we have $\tilde{\omega}([y],[z_k])=\tau(y^1,c_k^1) + \tau(c_k^2,z^2)$ and $\tilde{\omega}([y],[z])=\tau(y^1,z^1)$. The convergence of $([c_k])_{k \in \N}$ gives a limit jump point $[c]$ for $[z]$. On the one hand, we have 
$\lim_k \tilde{\omega}([y],[z_k])= \lim_k \tau(y^1,c_k^1) + \tau(c_k^2,z_k^2) = \tau(y^1,c^1) + \tau(c^2,z^2)=\tau(y^1,c^1)+\tau(c^1,z^1) \leq \tau(y^1,z^1)=\ttau([y],[z])=\tilde{\omega}([y],[z])$ by the maximality of $\tilde{\omega}$ or by the reverse triangle inequality for $\tau$. 
On the other hand, we know $\ttau$ is lower semi-continuous and $\ttau=\tilde{\omega}$, so
$\lim_k \tilde{\omega}([y],[z_k]) \geq \tilde{\omega}([y],[z])$. This establishes that $\tilde{\omega}$ is continuous and hence $I_X([p],[q])$ is a localizable neighbourhood of $[x]$.
\end{proof}

Collecting the previous propositions, we obtain the following theorem:

\begin{thm}[Amalgamation of length spaces]
Let $X_1$ and $X_2$ be two strongly causal and locally compact Lorentzian length spaces. Then $X:= \amA$ is a (strongly localizable) Lorentzian length space.
\end{thm}

\begin{proof}
All assumptions in Propositions \ref{prop: caus pc}, \ref{prop: intrinsic}, \ref{lcc} and \ref{prop: loc} are satisfied, so $X$ is a Lorentzian length space. 
The claim on strong localizability follows by Corollary \ref{llsbasis} and Lemma \ref{lem: str caus loc nhood basis}\footnote{Technically, for the strong localizability one needs to either slightly adapt the construction of neighbourhoods in Proposition \ref{lcc} or apply Proposition \ref{strong causality}.}.
\end{proof}

\section{The causal inheritance}
In this chapter, we take a look at the compatibility of the causal ladder with Lorentzian amalgamation. We consider various steps of the causal ladder in separate statements and then sum everything up in a final theorem. As it turns out, many causality conditions are able to be inherited by Lorentzian amalgamation if one is willing to admit some additional properties.
\begin{prop}[Chronology and causality]
\label{prop: chr and caus}
Let $X_1$ and $X_2$ be two Lorentzian pre-length spaces and $X:=\amA$ their amalgamation.
\begin{itemize}
\item[(i)] If $X_1$ and $X_2$ are chronological, then so is $X$.
\item[(ii)] If $X_1$ and $X_2$ are causal, then so is $X$.
\end{itemize}
\end{prop}
\begin{proof}
(i) This is immediate from Proposition \ref{prop: gluing properties}(i), as for any $[x] \in X_i$ we have $\ttau([x],[x])=\tau(x^i,x^i)=0$ since $X_i$ is chronological. \\

(ii) If $[x]$ and $[y]$ are from the same space, say $X_1$, then this follows from the causality of $X_1$. So assume $[x] \in X_1 \setminus A_1, [y] \in X_2 \setminus A_2$ with $[x] \tleq [y]$ and $[y] \tleq [x]$. Then we find $[a],[b] \in A$ such that 
\begin{equation}
\label{causal}
x^1 \leq a^1 \sim a^2 \leq y^2 \text{ and } y^2 \leq b^2 \sim  b^1 \leq x^1.
\end{equation}
By the transitivity of $\leq$ it follows that $b^1 \leq a^1$ and $a^2 \leq b^2$. 
Since $X_1$ and $X_2$ are causal and $f$ is $\leq$-preserving, we obtain $a^1=b^1$ and $a^2=b^2$, i.e., $[a]=[b]$. 
Then (\ref{causal}) implies $[x]=[a]=[b]=[y]$, a contradiction.
\end{proof}
\begin{prop}[Strong causality]
\label{strong causality}
Let $X_1$ and $X_2$ be two strongly causal Lorentzian pre-length spaces. Then $X:= X_1 \sqcup_A X_2$ is strongly causal.
\end{prop}
\begin{proof}
If $[z] \in X \setminus A$, let $U$ be a neighbourhood of $[z]$ that does not meet $A$. Then $\pi^{-1}(U)$ is a neighbourhood of, say, $z^1$ in $X_1$ that does not meet $A_1$. 
By strong causality of $X_1$, we find points $x^1_1,y_1^1, x_2^1,y_2^1, \ldots, x_n^1,y_n^1$ in $X_1$ such that $z^1 \in \cap_{i=1}^n I_1(x_i^1,y_i^1) \subseteq \pi^{-1}(U)$.
As $\pi|_{X_1}$ is injective, we compute
$[z] \in \pi(\cap_{i=1}^n I_1(x_i^1,y_i^1)) = \cap_{i=1}^n \pi(I_1(x_i^1,y_i^1)) =\cap_{i=1}^n I_X([x_i],[y_i]) \subseteq \pi(\pi^{-1}(U)) = U$. \\

Let $[z] \in A$ and let $U \subseteq X$ be a neighbourhood of $[z]$. Then $\pi^{-1}(U)$ is open in $X_1 \sqcup X_2$, i.e., $\pi^{-1}(U)=U_1 \sqcup U_2$, where $U_i$ is open in $X_i, i=1,2$.
Moreover, by definition we have $f(U_1 \cap A_1)=U_2 \cap A_2$. 
By strong causality, we find points $x_1^1,y_1^1,x_2^2,y_2^2, \ldots, x_n^1,y_n^1$ in $X_1$ and $p_1^2,q_1^2, p_2^2,q_2^2, \ldots, p_m^2,q_m^2$ in $X_2$ such that $z^1 \in D_1:=\cap_{i=1}^n I_1(x_i^1,y_i^1) \subseteq U_1$ and $z^2 \in D_2 := \cap_{j=1}^m I_1(p_j^2,q_j^2) \subseteq U_2$.
Then we find open neighbourhoods $V_i \subseteq D_i$ such that $f(V_1 \cap A_1) = V_2 \cap A_2$. 
As $A_i$ is non-timelike locally isolating, there exists $b_+^i, b_-^i \in V_i \cap A_i$, such that $b_-^i \ll z^i \ll b_+^i$. Since $D_i$ is causally convex, we have $I(b_-^i,b_+^i) \subseteq D_i \subseteq U_i$. 
Finally, $I(b_-^1,b_+^1) \sqcup I(b_-^2,b_+^2) \subseteq U_1 \sqcup U_2 = \pi^{-1}(U)$ and hence 
\begin{equation}
U=\pi(\pi^{-1}(U))=\pi(U_1 \sqcup U_2) \supseteq \pi(I(b_-^1,b_+^1) \sqcup I(b_-^2,b_+^2))=I_X([b_-],[b_+]), 
\end{equation}
where the first equality holds since $\pi$ is surjective and the last equality is due to Proposition \ref{causal rep}. To summarize, in an arbitrary neighbourhood of $[z]$ we found an intersection of elements of the subbasis of the Alexandrov topology, showing that $X$ is strongly causal.
\end{proof}

\begin{defin}[Time observation]
A subset $A$ of a Lorentzian pre-length space $A$ is called future observing if for all $x \in X$ there exists $a_- \in A$ such that $J^+(x) \cap A \subseteq J^+(a_-) \cap A$. Similarly, we define a past observing set. We say $A$ is time observing if it is both future and past observing. In particular, we then infer for all $x,y \in X$ the existence of $a_-, a_+ \in A$ such that $J(x,y) \cap A \subseteq J(a_-,a_+) \cap A$.
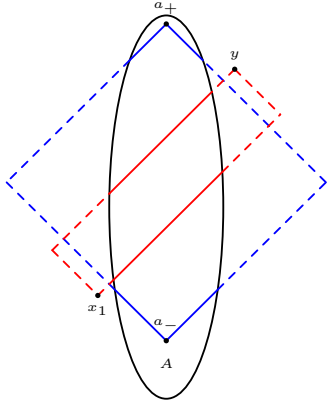
\begin{figure}[H]
\begin{center}
\definecolor{ffqqqq}{rgb}{1,0,0}
\definecolor{qqqqff}{rgb}{0,0,1}
\begin{tikzpicture}[line cap=round,line join=round,>=triangle 45,x=1cm,y=1cm, scale=1.5]
\draw [rotate around={90:(0,-0.01878591394685843)},line width=0.7pt] (0,-0.01878591394685843) ellipse (1.6942545779695612cm and 0.500030738833177cm);
\draw [line width=0.7pt,dashed,color=blue] (-1.4,0.2)-- (-0.45123870101592833,-0.7487612989840716);
\draw [line width=0.7pt,dashed,color=red] (-1,-0.4)-- (-0.49877471253355843,0.10122528746644155);
\draw [line width=0.7pt,color=blue] (-0.45123870101592833,-0.7487612989840716)-- (0,-1.2);
\draw [line width=0.7pt,color=qqqqff] (0,-1.2)-- (0.45123870101592833,-0.7487612989840716);
\draw [line width=0.7pt,dashed,color=blue] (0.45123870101592833,-0.7487612989840716)-- (1.4,0.2);
\draw [line width=0.7pt,dashed,color=blue] (1.4,0.2)-- (0.3216855996517911,1.2783144003482085);
\draw [line width=0.7pt,color=blue] (0.3216855996517911,1.2783144003482085)-- (0,1.6);
\draw [line width=0.7pt,color=qqqqff] (0,1.6)-- (-0.3216855996517911,1.2783144003482085);
\draw [line width=0.7pt,dashed,color=blue] (-0.3216855996517911,1.2783144003482085)-- (-1.4,0.2);
\draw [line width=0.7pt,dashed,color=ffqqqq] (-1,-0.4)-- (-0.6,-0.8);
\draw [line width=0.7pt,dashed,color=ffqqqq] (-0.6,-0.8)-- (-0.4625304321592599,-0.6625304321592599);
\draw [line width=0.7pt,color=ffqqqq] (-0.4625304321592599,-0.6625304321592599)-- (0.4915697953082756,0.29156979530827554);
\draw [line width=0.7pt,dashed,color=ffqqqq] (0.4915697953082756,0.29156979530827554)-- (1,0.8);
\draw [line width=0.7pt,dashed,color=ffqqqq] (0.6,1.2)-- (0.39961493275060683,0.9996149327506068);
\draw [line width=0.7pt,color=ffqqqq] (0.39961493275060683,0.9996149327506068)-- (-0.49877471253355843,0.10122528746644155);
\draw [line width=0.7pt,dashed,color=ffqqqq] (0.6,1.2)-- (1,0.8);

\begin{tiny}
\coordinate [circle, fill=black, inner sep=0.7pt, label=90: {$a_+$}] (p) at (0,1.6);
\coordinate [circle, fill=black, inner sep=0.7pt, label=270: {$x_1$}] (p) at (-0.6,-0.8);

\coordinate [circle, fill=black, inner sep=0.7pt, label=90: {$y$}] (p) at (0.6,1.2);

\coordinate [circle, fill=black, inner sep=0.7pt, label=90: {${a_{-}}$}] (p) at (0,-1.2);

\coordinate [label=270: {$A$}] (p) at (0,-1.3);
\end{tiny}
\end{tikzpicture}
\end{center}
\caption{Showcasing the property of being time observing.}
\end{figure}
Note that $J^-(A) \cap J^+(A)=X$ is a sufficient condition for $A$ being time observing. Indeed, in this case for all $x \in X$ we find $a_-$ and $a_+$ in $A$ such that $a_- \leq x \leq a_+$ and hence $J^-(x) \subseteq J^-(a_+)$ as well as $J^+(x) \subseteq J^+(a_-)$.
\end{defin}

\begin{prop}[Global hyperbolicity]
\label{prop: glob hyp}
Let $X_1$ and $X_2$ be two globally hyperbolic, causally path connected, locally causally closed and $d$-compatible Lorentzian pre-length spaces. Let $A_1$ and $A_2$ be time observing. Then $X:= \amA$ is globally hyperbolic.
\end{prop}

\begin{proof}
By Lemma \ref{causal ladder0}, $X_1$ and $X_2$ are strongly causal, so by Proposition \ref{strong causality}, $X$ is strongly causal. It is easy to infer via Proposition \ref{prop: loc} that $X$ is $d$-compatible as well. 
Then $X$ fulfills all the assumptions to apply Lemma \ref{causal ladder1}, showing that $X$ is non-totally imprisoning. 
It is left to show that diamonds in $X$ are compact.
Let $[x],[y] \in X$ with $[x] \tleq [y]$. 
If $[x],[y] \in X_i \setminus A_i$ and $J_i(x^i,y^i) \cap A_i = \emptyset$, then $J_X([x],[y])=\pi(J_i(x^i,y^i))$ inherits the compactness by assumption.
If $[x],[y] \in A$, the compactness of $J_X([x],[y])=\pi(J_1(x^1,y^1) \sqcup J_2(x^2,y^2))$ follows by assumption as well. 
The two cases left to show are $[x],[y] \in X_i$ and their causal diamond meets $A$ or, say, $[x] = \{x^1\}, [y] = \{y^2\}$. \\

Assume we are in the first case, and both points are contained in, say, $X_1 \setminus A_1$. 
By Proposition \ref{causal rep}, we then have $J_X([x],[y])=\pi(J_1(x^1,y^1) \sqcup (\cup_{x^1 \leq a^1}J_2^+(a^2)) \cap (\cup_{b^1 \leq y^1} J_2^-(b^2)))$.
Let $([p_n])_{n \in \N}$ be a sequence in $J_X([x],[y])$. 
If there exists a subsequence $(p_{n_k}^1)_{k \in \N}$ in $J_1(x^1,y^1)$, then by the compactness of $J_1(x^1,y^1)$ we can extract a convergent subsequence which yields a convergent subsequence of $([p_n])_{n \in \N}$. 
So suppose all (but finitely many) sequence members are from the second space originally. 
By definition, this means that for all large enough $n \in \N$ we find $a_n^1, b_n^1 \in A_1$ such that $a_n^2 \leq p_n^2 \leq b_n^2$. 
In particular, $x^1 \leq a_n^1 \leq b_n^1 \leq y^1$, so $a_n^1,b_n^1 \in J_1(x^1,y^1) \cap A_1$. 
Since $A_1$ is time observing, we find $q_-^1,q_+^1 \in A_1$ such that $J_1(x^1,y^1) \cap A_1 \subseteq J_1(q_-^1,q_+^1) \cap A_1$. 
Then also $a_n^1,b_n^1 \in J_1(q_-^1,q_+^1) \cap A_1$ and since $f$ is causality preserving, we conclude $a_n^2,b_n^2 \in J_2(q_-^2,q_+^2) \cap A_2$ as well. 
Thus, $(p_n^2)_{n \in \N}$ is contained in a compact set and has a convergent subsequence, say without loss of generality the whole sequence already converges to some $p^2 \in \widebar{\cup_{x^1 \leq a^1}J_2^+(a^2)) \cap (\cup_{b^1 \leq y^1} J_2^-(b^2))} =: \widebar{D}$. \\

It is left to show that $p^2$ is not only in the closure but inside the set $D$ itself. To this end consider the resulting sequences $(b_n^2)_{n \in \N}$ and $(a_n^2)_{n \in \N}$. Clearly, $x^1 \leq a_n^1 \leq b_n^1 \leq y^1$ for all $n$ by assumption.
The two sequences $(b_n^1)_{n \in \N}$ and $(a_n^1)_{n \in \N}$ also possess convergent subsequences since they are contained in the compact set $J_1(x^1,y^1)$. 
Say $a_n^1 \to a^1, b_n^1 \to b^1$, then $x^1 \leq a^1$ and $b^1 \leq y^1$ by the (global) causal closedness of $X_1$. 
Thus, $J_2^+(a^2)$ and $J_2^-(b^2)$ are also part of the unions in $D$.
By the causal closedness of $X_2$, we infer from $a_n^2 \leq p_n^2 \leq b_n^2$ that $a^2 \leq p^2 \leq b^2$. So $p^2$ is an element of $D$ and not just of $\widebar{D}$. 
In particular, $[p] \in J_X([x],[y])$.
This shows that $J_X([x],[y])$ is compact. 
The case of $[x] \in X \setminus A, [y] \in A$ is similar and less complicated than the one we just proved. \\

Finally, consider the case $[x]=\{x^1\}, [y]=\{y^2\}$. 
Then  $J_X([x],[y])=\pi(J^+_1(x^1) \cap (\cup_{y^2 \gg a^2} J_1^-(a^1)) \sqcup 
(\cup_{x^1 \ll b^1} J_2^+(b^2)) \cap J_2^-(y^2))$. 
Let $([p_n])_{n \in \N}$ be a sequence in $J_X([x],[y])$, then there is a subsequence in one of the two original spaces. Say without loss of generality $p_n^1 \in J^+_1(x^1) \cap (\cup_{y^2 \gg a^2} J_1^-(a^1))$ for all $n$. 
Then for all $n$ we have that there exists $[a_n] \in A: x^1 \leq p_n^1 \leq a_n^1 \sim a_n^2 \leq y^2$. As $A_2$ is time observing, we find $c^2 \in A_2$ such that $J_2^-(c^2) \cap A_2 \supseteq J_2^-(y^2) \cap A_2$. Thus, we have $c^2 \geq a_n^2$ for all $n$, and hence $c^1 \geq a_n^1$ as well. 
In particular, $p_n^1, a_n^1 \in J_1(x^1,c^1)$ for all $n$, so we infer the existence of converging subsequences, say $p_n^1 \to p^1$ and $a_n^1 \to a^1$. By the global causal closedness, it follows that $p^1 \leq a^1$. Moreover, by $y^2 \geq a_n^2$ for all $n$ we infer $y^2 \geq a^2$, so $p^1 \in J_1^+(x^1) \cap (\cup_{y^2 \gg a^2} J_1^-(a^1))$ and hence $[p] \in J_X([x],[y])$, establishing that diamonds are compact.
\end{proof}
\begin{prop}[Distinction]
\label{causal distinction}
Let $X_1$ and $X_2$ be two strongly causal, non-timelike locally isolating and distinguishing Lorentzian pre-length spaces. Then $X:=\amA$ is distinguishing.
\end{prop}
\begin{proof}
We need to show that $I_X^{\pm}([x])=I_X^{\pm}([y])$ implies $[x]=[y]$ for all $[x],[y] \in X$. We distinguish several cases depending on where the points originally come from.
We will only consider the future case, as the past case is completely analogous.
Suppose first $[x],[y] \in A$ with $I^{+}([x])=I_X^{+}([y])$. By Proposition \ref{causal rep}, we then have $\pi(I_1^+(x^1) \sqcup I_2^+(x^2))=\pi(I_1^+(y^1) \sqcup I_2^+(y^2))$.
As $f$ is $\ll$-preserving, we clearly have $\pi(I^+_1(x^1) \cap A_1)=\pi(I_2^+(x^2) \cap A_2)$ as well as $\pi(I^+_1(y^1) \cap A_1)=\pi(I^+_2(y^2) \cap A_2)$.
By Lemma \ref{recovering subsets}, we then infer 
$I^+_1(x^1) \sqcup I^+_2(x^2) = \pi^{-1}(\pi(I^+_1(x^1) \sqcup I^+_2(x^2))) = \pi^{-1}(\pi(I^+_1(y^1) \sqcup I^+_2(y^2))) = I^+_1(y^1) \sqcup I^+_2(y^2)$.
So in particular $I^+_1(x^1)=I^+_1(y^1)$ and $I^+_2(x^2)=I^+_2(y^2)$, hence $[x]=[y]$ by $X_1$ and $X_2$ being distinguishing. \\

Now suppose, say, $[x] \in X_1$ and $[y] \in X_1 \setminus A_1$, with $I_X^{+}([x])=I_X^{+}([y])$.
If $[x] \in X_1 \setminus A_1$, then we have
$\pi(I_1^+(x^1) \sqcup (\cup_{x^1 \ll a^1} I_2^+(a^2))=\pi(I_1^+(y^1) \sqcup (\cup_{y^1 \ll b^1} I_2^+(b^2))$ by Proposition \ref{causal rep}.
Similarly as above, we have $\pi(I_1^+(x^1) \cap A_1)=\pi(\cup_{x^1 \ll a^1} I_2^+(a^2) \cap A_2)$ by the $\ll$-preservation of $f$. 
Indeed, let $b^1 \in I_1^+(x^1) \cap A_1$, then since $\ll$ is open, there is a neighbourhood $U$ of $b^1$ such that $U \subseteq I_1^+(x^1)$. 
As $A_1$ is non-timelike locally isolating, we find $b_-^1 \in U \cap A_1$ such that $b_-^1 \ll b^1$. 
Then also $b_-^2 \ll b^2$ and since $x^1 \ll b_-^1$, it follows that $b^2 \in \cup_{x^1 \ll a^1} I_2^+(a^2) \cap A_2$.
Conversely, if $b^2 \in \cup_{x^1 \ll a^1} I_2^+(a^2) \cap A_2$, then we obtain the existence of $c^1 \in I_1^+(x^1) \cap A_1$ such that $c^2 \ll b^2$. Then also $c^1 \ll b^1$ and hence $x^1 \ll b^1$, establishing $b^1 \in I_1^+(x^1) \cap A_1$. 
In summary, we have shown $b^1 \in I_1(x^1) \cap A_1$ whenever $b^2 \in \cup_{x^1 \ll a^1} I_2^+(a^2) \cap A_2$, i.e., $\pi(I_1^+(x^1) \cap A_1)=\pi(\cup_{x^1 \ll a^1} I_2^+(a^2) \cap A_2)$. 
Then we have $I_1^+(x^1) \sqcup (\cup_{x^1 \ll a^1} I_2^+(a^2)=I_1^+(y^1) \sqcup (\cup_{y^1 \ll b^1} I_2^+(b^2)$ by Lemma \ref{recovering subsets}. So $I_1^+(x^1)=I_1^+(y^1)$ and hence $x^1=y^1$ and $[x]=[y]$ by $X_1$ being distinguishing. \\

In the case of $[y] \in X_1 \setminus A_1$ and $[x] \in A$, we can proceed analogously using $\pi(I_1^+(x^1) \sqcup I_2^+(x^2))=\pi(I_1^+(y^1) \sqcup (\cup_{y^1 \ll b^1} I_2^+(b^2))$. \\

It is left to cover the case of $[x] \in X_1 \setminus A_1$ and $[y] \in X_2 \setminus A_2$. We then have $I_X^+([x])=\pi(I_1^+(x^1) \sqcup \cup_{x^1 \ll a^1} I_2^+(a^2))=\pi(\cup_{y^2 \ll a^2} I_1^+(a^1) \sqcup I_2^+(y^2)) = I_X^+([y])$. 
By similar arguments as before we conclude $I_1^+(x^1)=\cup_{y^2 \ll a^2} I_1^+(a^1)$ and $\cup_{x^1 \ll a^1} I_2^+(a^2)=I_2^+(y^2)$.
Note that using the same letter in both unions is in fact not sloppy but accurate: if, say, $y^2 \ll a^2$, then by the above equality we infer the existence of some $b^1 \in I_1^+(x^1)$ such that $b^2 \ll a^2$. Then $b^1 \ll a^1$ and since $x^1 \ll b^1$, we have $x^1 \ll a^1$. We can apply the same argument vice versa.
To emphasize, we have $x^1 \ll a^1$ if and only if $y^2 \ll a^2$ for all $[a] \in A$.
Now we use similar arguments as in Proposition \ref{strong causality}. Since $x^1 \in X_1 \setminus A_1$, we find a neighbourhood $U_1$ containing $x^1$ that does not meet $A_1$. By Lemma \ref{lem: nlti basis} we find a timelike diamond $D_1$ inside $U_1$ which contains $x^1$ and whose endpoints lie in $U_1$.
Then for all $q^1 \in D_1 \cap I_1^+(x^1)$ (which is nonempty by non-timelike local isolation), there cannot exist $a^1 \in I_1^+(x^1) \cap A_1$ such that $a^1 \ll q^1$, since otherwise $a^1 \in D_1$. In particular, $q^1 \notin \cup_{y^2 \ll a^2} I_1^+(a^1)$ while $q^1 \in I_1^+(x^1)$, a contradiction. So under our assumptions, two points from two different spaces cannot have the same future.
In summary, we conclude that $X$ is distinguishing.
\end{proof}

\begin{ex}[Double flagpole]
\label{ex: flagpole}
The following example demonstrates that without the assumptions of strong causality and non-timelike local isolation, Proposition \ref{causal distinction} fails. Consider two ``flagpoles'' in the Minkowski plane, i.e., the union of (part of) the $t$-axis with a rectangle, see Figure \ref{fig: flagpole}. Both of these spaces are distinguishing. Gluing them along the $t$-axis results in a space which is not distinguishing anymore.
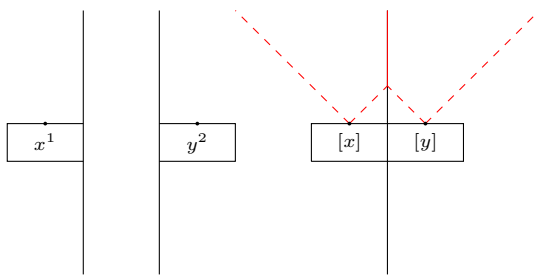
\begin{figure}
\begin{center}
\begin{tikzpicture}
\draw (0,-0.5) -- (0,3);
\draw (0,1) -- (1,1)--(1,1.5)--(0,1.5);

\draw (-1,-0.5) -- (-1,3);
\draw (-1,1) -- (-2,1)--(-2,1.5)--(-1,1.5);

\draw (3,-0.5) -- (3,3);
\draw (2,1) -- (4,1) -- (4,1.5) -- (2,1.5) -- (2,1);

\draw[red,dashed] (3.5,1.5) -- (3,2);
\draw[red,dashed] (3.5,1.5) -- (5,3);
\draw[red] (3,2) -- (3,3);
\draw[red,dashed] (2.5,1.5) -- (3,2);
\draw[red,dashed] (2.5,1.5) -- (1,3);

\begin{scriptsize}
\coordinate [circle, fill=black, inner sep=0.5pt, label=270: {$y^2$}] (A1) at (0.5,1.5);

\coordinate [circle, fill=black, inner sep=0.5pt, label=270: {$x^1$}] (A1) at (-1.5,1.5);

\coordinate [circle, fill=black, inner sep=0.5pt, label=270: {$[y]$}] (A1) at (3.5,1.5);

\coordinate [circle, fill=black, inner sep=0.5pt, label=270: {$[x]$}] (A1) at (2.5,1.5);


\end{scriptsize}
\end{tikzpicture}
\end{center}
\caption{In this example, we have $I_X([x])=I_X([y])$ but $[x] \neq [y]$.}
\label{fig: flagpole}
\end{figure}
\end{ex}

Until now, most of the properties were able to be transferred either as is or with relatively harmless additional assumptions.
This changes now, as the remaining steps of the causal ladder appear to be less suitable for preservation via gluing constructions. \\

Let us begin with non-total imprisonment.
It is easily shown that a compact set in $X$ is built from two compact sets in $X_1$ and $X_2$, i.e., for any compact $K \subseteq X$ we have $K=\pi(K_1 \sqcup K_2)$ with $K_i \subseteq X_i, i=1,2$ compact.
Further, any causal curve in $K$ can thus be decomposed into causal curves in $K_1$ and $K_2$.
Then the first problem is that the causal curve may consist of infinitely many pieces.
At first glance one might eliminate this problem by imposing causal path-connectedness on $X_1$ and $X_2$, which would enable us to connect all pieces in one space into a single curve. 
Then of course the problem is that these connecting pieces cannot be guaranteed to stay inside $K_1$ or $K_2$.
This is yet another indicator that below strong causality, there is no way of taming the causality and relating it to the topology. \\

The following definition 
can be regarded as a ``non-intrinsic'' formulation of non-total imprisonment. We will show that under this stronger condition, the property of non-total imprisonment is inherited.

\begin{defin}[Extrinsic non-total imprisonment]
\label{def enti}
Let $X$ be a Lorentzian pre-length space. $X$ is called extrinsically non-totally imprisoning if for all compact sets $K \subseteq X$ there is a constant $C>0$ such that for all causal sequences $(x_1,x_2,\ldots, x_n)$ in $K$ we have $\sum_{i=1}^{n} d(x_i,x_{i+1}) \leq C$, where a causal sequence satisfies $x_i \leq x_{i+1}$ for all $i$. 
Since $L_d$ is given as the supremum of causal sequences, we immediately get that this property implies non-total imprisonment. 
Moreover, we also have the same bound for infinite sequences, since if there exists an infinite causal sequence whose sum of distances is larger than $C$, then this must already occur after summing up finitely many distances.
\end{defin}

\begin{rem}[Placing the extrinsic version on the causal ladder]
Recall the classical example of a spacetime which is causal but not strongly causal, cf.\ \cite[Example 14.12]{ON83}: the Lorentz cylinder $M:=\Sph^1 \times \R$ with two horizontal strips removed. This spacetime is also non-totally imprisoning, as in fact \emph{any} causal curve has a (uniformly) bounded $d$-length since wrapping around is forbidden via the removed pieces. 
Thus, also any causal chain has uniformly bounded $d$-length. So $M$ is an example of a extrinsically non-totally imprisoning space which is not strongly causal. This suggests at least that this property is not stronger than strong causality.
\end{rem}

\begin{lem}[Decomposition of compact subsets in amalgamation]
\label{lem: cpt sets}
Let $X_1$ and $X_2$ be metric spaces, $A_i \subseteq X_i$ closed and $f:A_1 \to A_2$ a locally bi-Lipschitz homeomorphism. Consider the quotient space $X$ of $X_1 \sqcup X_2$ generated by the equivalence relation $a \sim f(a)$ for all $a \in A_1$.
Let $K \subseteq X$ be compact. Then there exist compact $K_i \subseteq X_i, i=1,2$ such that $\pi^{-1}(K)=K_1 \sqcup K_2$.
\end{lem}

\begin{proof}
Let $K \subseteq X$ be compact. 
Then by definition $\pi^{-1}(K)=K_1 \sqcup K_2 \subseteq X_1 \sqcup X_2$ for some $K_i \subseteq K$.
Let, say, $(x^1_n)_{n \in \N}$ be a sequence in $K_1$. Then $([x_n])_{n \in \N}$ is a sequence in $K$, which contains a convergent subsequence as $K$ is compact, say $([x_n])_{n \in \N} \to [x] \in K$. Note that $\pi(X_1)$ is closed in $X$, hence $[x] \in K \cap \pi(X_1) = \pi(K_1)$. 
Thus, 
$x^1_n \to x^1$ follows easily, cf.\ Remark \ref{rem: convergence in amalg}. This shows that $K_1$ is compact and similarly for $K_2$.
\end{proof}

\begin{prop}[Non-total imprisonment]
\label{prop: nti}
Let $X_1$ and $X_2$ be two extrinsically non-totally imprisoning Lorentzian pre-length spaces. Then $X$ is non-totally imprisoning.
\end{prop}

\begin{proof}
Let $K$ be a compact subset of $X$. Then $\pi^{-1}(K)=K_1 \sqcup K_2$ for compact $K_1 \subseteq X_1$ and $K_2 \subseteq X_2$.
Let $\gamma: [a,b] \to K$ be a causal curve. 
Then either $\pi^{-1}(\gamma([a,b]))$ is a causal curve in $K_i$ or it consists of pieces of causal curves in $K_1$ and $K_2$. 
In the first case we can take the original constant $C_i$ of $K_i$ by the extrinsic non-total imprisonment of $X_i$, so assume we are in the second case. 
Without loss of generality, assume $\gamma$ starts out in $X_1$, leaves $X_1$ through $[p]=\gamma(t_1) \in A$ and enters $X_1$ again at $[q]=\gamma(t_2) \in A$. Then clearly $p^1 \leq q^1$. By the extrinsic non-total imprisonment of $X_1$, we have $L_d(\gamma|_{[a,t_1]}) + d(p^1,q^1) \leq C_1$ (remember that $L_d$ is given as the supremum of lengths of causal chains, each of which satisfies the given bound by assumption). 
Doing this iteratively for all jump points, we end up with two causal sequences (each of which may have infinitely\footnote{Notice that at there are at most countably infinitely many jump points. Indeed, the domain $[a,b]$ can only be covered by countably many non-trivial subintervals, as, e.g., a measure theoretic argument shows.} many members, as $\gamma$ may jump infinitely often) in $K_1$ and $K_2$, respectively. In the end, this implies $L_{\td}(\gamma) \leq C_1 + C_2$.
\end{proof}

Finally, we collect all compatibility results regarding the causal ladder and amalgamation into the following theorem.
\begin{thm}[Causal inheritance]
Let $X_1$ and $X_2$ be two Lorentzian pre-length spaces and $X:= \amA$ their Lorentzian amalgamation. Then we have the following preservation of causality conditions.
\begin{itemize}
\item[(i)] If $X_1$ and $X_2$ are chronological, then so is $X$.

\item[(ii)] If $X_1$ and $X_2$ are causal, then so is $X$.
\item[(iii)] If $X_1$ and $X_2$ are extrinsically non-totally imprisoning, then $X$ is non-totally imprisoning.
\item[(iv)] If $X_1$ and $X_2$ are strongly causal, then so is $X$.
\item[(v)] If $X_1$ and $X_2$ are strongly causal, non-timelike locally isolating and distinguishing, then $X$ is distinguishing.
\item[(vi)] If $X_1$ and $X_2$ are causally path-connected, locally causally closed, d-compatible and globally hyperbolic with $A_1$ and $A_2$ time observing, then $X$ is globally hyperbolic.
\end{itemize}
\end{thm}

\begin{proof}
This is a summary of Propositions \ref{prop: chr and caus}, \ref{prop: nti}, \ref{strong causality}, \ref{causal distinction} and \ref{prop: glob hyp}.
\end{proof}
\section{Outlook on missing properties}
The remaining properties of the causal ladder, for now, are not inherited by gluing under ``reasonable'' additional assumptions. These are $K$-causality, reflectivity and causal simplicity.
Nevertheless, we will give our (unsuccessful) thoughts and discuss these properties in some detail. 

\subsection*{$K$-causality}
Let us begin with $K$-causality. In \cite{BGH21}, it is shown that for certain Lorentzian pre-length spaces this is equivalent to the existence of a time function, which seems a bit more convenient to work with.
The most obvious approach for constructing a time function on a glued space $X$ under the assumption that $X_1$ and $X_2$ do possess time functions is the following: start with a time function $T_1:A_1 \to \R$ and consider the corresponding time function $T_2:= T_1 \circ f^{-1} : A_2 \to \R$. Suppose we could extend $T_1$ and $T_2$ to time functions defined on the entirety of the spaces $X_1$ and $X_2$ (and denote the extensions again by $T_1$ and $T_2$, respectively). 
Then consider a function $T: X \to \R$ defined by $T([x]):=T_i(x^i)$. This function is well defined since for $[x] \in A$ we have $T_1(x^1)=T_2(x^2)$. If $[x] \tl [y]$ in $X$, then either $x^i < y^i$ or there exists $[a] \in A$ such that $x^i < a^i \sim a^j < y^j$. In the first case we have $T([x])=T_i(x^i) < T_i(y^i)=T([y])$, since $T_i$ is a time function on $X_i$.
In the second case, we know $T([x])=T_i(x^i) < T_i(a^i)=T_j(a^j) < T_j(y^j)=T([y])$. Thus, $T$ is a time function (it is continuous since $T_1$ and $T_2$ are). \\

Of course, the difficulty lies in extending a time function from a subset to the whole space in such a way that it stays a time function. 
Historically, techniques and results from utility theory have enabled an entirely different approach to showing the existence of a strictly increasing function with respect to the $K$-relation compared to original results \cite{BH69, Ger70, EH73} (see \cite[Page 2]{Min10} for a detailed discussion).
Most importantly, \cite{Min10} noted that the existence of time functions on $K$-causal spacetimes is an immediate consequence of Levin's theorem:
\begin{thm}[Levin's theorem, \cite{Lev83}]
Let $X$ be a second countable and locally compact topological Hausdorff space and $\leq$ a closed, reflexive and transitive relation on $X$. Then there exists a continuous function $T:X \to \R$ such that $x \leq y \Rightarrow T(x) \leq T(y)$ with equality if and only if $y \leq x$ as well.
\end{thm}

Let for now $X$ be a topological space equipped with a preorder, i.e., a reflexive and transitive relation, denoted by $\leq$.
For the task of extending a merely increasing function\footnote{Here, $f:X \to \R$ is increasing if $x \leq y \Rightarrow f(x) \leq f(y)$ holds for all $x,y \in X$. 
In the context of order topology and/or utility theory, such functions are called isotone.} from a subset to the whole space $X$, \cite{Min13a} gives the following necessary and sufficient criterion, also found in \cite[Theorem 2]{Nac65}. To this end, recall that a subset $A$ is called decreasing if $x \in A, y \leq x \Rightarrow y \in A$, and an increasing set is defined analogously. Then $D(A)$ denotes the smallest closed decreasing subset that contains $A$. Similarly, $I(A)$ denotes the smallest closed increasing subset containing $A$. Finally, a normally preordered space is a topological space equipped with a preorder such that whenever $A$ is closed and decreasing and $B$ is closed and increasing and $A \cap B = \emptyset$, then there exist open sets $A' \supseteq A, B' \supset B$ such that $A'$ is decreasing and $B'$ is increasing and $A' \cap B' = \emptyset$. 

\begin{thm}[Extension criterion, {\cite[Theorem 3.1]{Min13a}}]
Let $X$ be a normally preordered space and let $A$ be a subspace. Let $T:A \to [0,1]$ be a continuous isotone function. Then $T$ possesses an extension to a continuous isotone function on $X$ if and only if the following holds for all $s,t \in [0,1]$:
\begin{equation}
\label{eq: extension}
s < t \Rightarrow D(T^{-1}([0,s]) \cap I(T^{-1}([t,1])) = \emptyset.
\end{equation}
\end{thm}
Two problems become apparent rather quickly: on the one hand, denoting the $K$-relation by $\leq^K$, we do not know whether composing a $\leq^K$-increasing function on $A_1$ with $f^{-1}$ yields a $\leq^K$-increasing function on $A_2$. For this one really needs to take into consideration the topology and causality of the whole spaces and not just of the homeomorphic subsets. This is why it is expected that requiring $f$ to be $\leq^K$-preserving is strictly necessary. 
On the other hand, time functions are not just isotone, rather they are strictly increasing with respect to $\leq$, i.e., if $x < y$\footnote{Notice that in this setting, $x < y$ is defined as $x \leq y$ and $y \not\leq x$. If the relation is antisymmetric, this is equivalent to $x \leq y$ and $x \neq y$.} then $f(x) < f(y)$. \\

In this regard there are also some fundamental topological results: 
to extend a strictly increasing function in the same fashion as above, \cite[Theorem 2.1]{Hus18}\footnote{To not distract too much with new notation which is then not really used anyways, we decided to only give a reference in this case.} demands, among other things, a partial order $\leq$ to be order-separable, which includes that there exists a countable subset $C$ of $X$ such that for all $x,z \in X$ with $x < z$ there is $y \in C$ such that $x < y <  z$.
By considering the disjoint and uncountable collection of lightlike segments from $(0,a)$ to $(1,a+1)$ for $a \in [0,1]$ in the Minkowski plane, one sees that the causal relation $\leq$ (and hence also the $K$-relation) is not order-separable even in very nice spaces. Thus, this is an unreasonable requirement. 
Next, one could try to apply this theory to the timelike relation, which is order-separable if $X$ is separable: indeed, $I(x,y)$ is open for all $x,y \in X$, and separability yields the existence of a countable subset $\{p_i\}_{i \in \N}$ such that each nonempty open set in $X$ contains some $p_i$. This would at least lead to the existence of a semi-time function\footnote{A semi-time function is a continuous function $T: X \to \R$ satisfying $x \ll y \Rightarrow T(x) < T(y)$.} on a glued space. However, $\ll$ is neither closed nor a preorder.
In summary, it seems that there is currently no machinery available to guarantee the extension of a (semi-)time function. \\

As a small consolation, one can establish the existence of a continuous isotone function on an amalgamation. This is done by imposing the topological conditions required in Levin's theorem on the individual spaces and observe that these are inherited. 
Notice that we now switch back to our usual notation of Lorentzian pre-length spaces $X_1$ and $X_2$ and their amalgamation $X$.
\begin{prop}[Existence of isotone functions]
Let $X_1$ and $X_2$ be two locally compact and second countable Lorentzian pre-length spaces.
Then $X$ admits a continuous $\leq$-isotone function.
\end{prop}
\begin{proof}
We only need to show that $X$ is second countable and locally compact, to then apply Levin's theorem to the $K$-relation on $X$. To this end, recall that being second countable and Lindelöf is equivalent for metric spaces. So let $\{U^i\}_{i \in I}$ be an open cover of $X$. Then $\pi^{-1}(U^i)=U_1^i \sqcup U_2^i$ is open in $X_1 \sqcup X_2$, i.e., $U_1^i$ and $U_2^i$ are open in $X_1$ and $X_2$, respectively, for all $i \in I$. 
Moreover, $\{U_1^i\}_{i \in I}$ and $\{U_2^i\}_{i \in I}$ are open covers of $X_1$ and $X_2$, respectively. Thus, there exist countable subcovers $\{U_1^{i_j}\}_{j \in \N}$ and $\{U_1^{i_k}\}_{k \in \N}$ of $X_1$ and $X_2$, respectively. 
In particular, $\{U_1^{i_j} \sqcup U_2^{i_j}\}_{j \in \N} \cup \{U_1^{i_k} \sqcup U_2^{i_k} \}_{k \in \N}$ is a countable subcover of $X_1 \sqcup X_2$ such that the projection of each set is open in $X$. This yields a countable subcover of $X$, establishing that $X$ is Lindelöf and hence second countable. \\

Concerning local compactness, let first, say, $[x] \in X_1 \setminus A_1$. 
Let $U_1$ be a compact neighbourhood of $x^1$ which does not meet $A_1$ and let $V_1 \subseteq U_1$ be an open set containing $x^1$. Then $\pi(U_1)$ is a compact neighbourhood of $[x]$. It is compact since $U_1$ is compact and $\pi$ is continuous, and $[x] \in \pi(V_1)$ is open since $\pi^{-1}(\pi(V_1))=V_1$ is open. 
The case of $[x] \in A$ is similar: we then find compact neighbourhoods $U_i$ of $x^i$ in $X_i, i=1,2$ such that $f(U_1 \cap A_1)=U_2 \cap A_2$. 
Let $V_i \subseteq U_i$ be open sets containing $x^i$ such that $f(V_1 \cap A_1)=V_2 \cap A_2$. Then $\pi(U_1 \sqcup U_2)$ is a compact neighbourhood of $[x]$. It is compact since $\pi$ is continuous and $[x] \in \pi(V_1 \sqcup V_2)$ is open since $\pi^{-1}(\pi(V_1 \sqcup V_2))=V_1 \sqcup V_2$ is open, cf.\ Lemma \ref{recovering subsets}.
Thus, $X$ satisfies the assumptions for Levin's theorem for the $K$-relation.
\end{proof}

\subsection*{Reflectivity}
The next step on the ladder is causal continuity, or more precisely, the reflectivity part of this property as we already showed that being distinguishing is inherited. 
Our by now standard approach of distinguishing cases of where the points originally come from starts out very promising, as all cases except both points being isolated in different spaces follow very naturally. 
Indeed, in all these other cases the inclusions $I_X^+([x]) \subseteq I_X^+([y])$ lead back to, say, $I_1^+(x^1) \subseteq I_1^+(y^1)$, and from this the other, more complicated side can be resolved as well. 
For instance, consider the case $[x],[y] \in X_1 \setminus A_1$, then via Proposition \ref{causal rep} and Lemma \ref{recovering subsets}, $I_X^+([x]) \subseteq I_X^+([y])$ yields $I_1^+(x^1) \subseteq I_1^+(y^1)$ and $\cup_{x^1 \ll a^1} I_2^+(a^2) \subseteq \cup_{y^1 \ll b^1} I_2^+(b^2)$. By the reflectivity of $X_1$ we obtain $I_1^-(y^1) \subseteq I_1^-(x^1)$. 
In particular, for any $b^1 \ll y^1$ it then follows that $b^1 \ll x^1$, so $\cup_{b^1 \ll y^1} I_2^-(a^2) \subseteq \cup_{a^1 \ll x^1} I_2^-(b^2)$. This implies $I_X^-([y]) \subseteq I_X^-([x])$. 
Thus, the only case left to consider is $[x]=\{x^1\}, [y]=\{y^2\}$, for which $I_X^+([x]) \subseteq I_X^+([y])$ yields $I_1^+(x^1) \subseteq U_{y^2 \ll b^2} I_1^+(b^1)$ and $\cup_{x^1 \ll a^1}I_2^+(a^2) \subseteq I_2^+(y^2)$. 
The seemingly easier half, in this case in $X_1$, can be achieved with the additional assumption that $\tau$ is continuous and $X_1$ and $X_2$ are non-timelike locally isolating. 
That is, under these assumptions one can infer $\cup_{b^2 \ll y^2} I_1^-(b^1) \subseteq I_1^-(x^1)$. 
Indeed, given $p^1 \in \cup_{b^2 \ll y^2} I_1^-(b^1)$ we find $[b] \in A$ such that $p^1 \ll b^1 \sim b^2 \ll y^2$. Consider a sequence $x_n^1 \to x^1, x_n^1 \in I_1^+(x^1)$\footnote{This exists by non-timelike local isolation: choose $x_n^1 \in I^+(x^1) \cap B_{\frac{1}{n}}(x^1)$.}. Since $I_1^+(x^1) \subseteq U_{y^2 \ll b^2} I_1^+(b^1)$ by assumption, for all $x_n^1$ we find $[a_n] \in A$ such that $x_n^1 \gg a_n^1 \sim a_n^2 \gg y^2$. 
In particular, $\tau(p^1,x_n^1) \geq \tau(p^1,b^1) + \tau(b^1,a_n^1) + \tau(a_n^1,x_n^1) > \tau(p^1,b^1) =: K > 0$.
As $\tau$ is continuous, we get $\tau(p^1,x^1) \geq K$, i.e., $p^1 \in I_1^-(x^1)$. 
Continuity of the time separation function implies that a space is reflecting, see \cite[Proposition 3.17]{ACS20}, so this is already a comparatively strong assumption. 
However, if one would follow the same approach, the other desired inclusion, $I_2^-(y^2) \subseteq \cup_{a^1 \ll x^1} I_2^-(a^2)$, suggests that one would have to be able to enclose a sequence of jump points into a compact subset to infer the existence of a limit which serves as a jump point for $x^1$. More precisely, given a sequence $x_n^1 \to x^1, x_n^1 \in I_1^+(x^1)$ and some $p^2 \in I_2^-(y^2)$, we find $[b_n]$ such that $p^2 \ll b_n^2 \sim b_n^1 \ll x_n^1$ for all $n$, but there is no lower bound on $\tau(b_n^1,x_n^1)$ and we do not know whether $([b_n])_{n \in \N}$ converges. 
To guarantee this, one would essentially be already on the level of global hyperbolicity, at which point the inheritance of being reflective would be redundant.
Finally, note that Example \ref{ex: flagpole} also serves as a counterexample for the inheritance of reflectivity, only in this case the omitted assumptions do not appear to be useful.

\subsection*{Causal simplicity}
At last, let us briefly discuss causal simplicity. This is similar to reflectivity in the sense that some compactness argument is required, which then requires the original spaces to be not far removed from global hyperbolicity. 
Indeed, showing the closedness of, say, $J_X^+([x])=\pi(J_1^+(x^1) \sqcup \cup_{x^1 \leq a^1}J_2^+(a^2))$, requires a subsequence of jump points $a_n^i$ to converge. More precisely, let $([p_n])_{n \in \N}$ be a sequence in $J_X^+([x])$ converging to $[p]$. If there is a subsequence $p_{n_k}^1 \to p^1$ in $X_1$, then the desired relation follows by the closedness of $J_1^+(x^1)$. 
So suppose that for all (large enough) $n$ there exists $[a_n]$ such that $x^1 \leq a_n^1 \sim a_n^2 \leq p_n^2$. Convergence of $([a_n])_{n \in \N}$ appears to be the only way of constructing $[a]$ such that $x^1 \leq a^1 \sim a^2 \leq p^2$. The most obvious assumptions that guarantee this are precisely the ones used in Proposition \ref{prop: glob hyp}.

\begin{chapt*}{Acknowledgments}
I want to thank Tobias Beran for valuable input throughout various stages of this project. \\

This work was supported by research grant P33594 of the
Austrian Science Fund FWF.
\end{chapt*}

\bibliographystyle{plain}
\bibliography{Biblio}

\begin{thebibliography}{10}

\bibitem{ACS20}
L.~{Aké Hau}, A.J. {Cabrera Pacheco}, and D.A. Solis.
\newblock On the causal hierarchy of {L}orentzian length spaces.
\newblock {\em Classical Quantum Gravity}, 37(21):215013, 2020.

\bibitem{AGKS19}
S.~B. Alexander, M.~Graf, M.~Kunzinger, and C.~S{ä}mann.
\newblock {Generalized cones as {L}orentzian length spaces: Causality,
  curvature, and singularity theorems}.
\newblock {\em Communicatios in {A}nalysis and {G}eometry}, 2019.
\newblock to appear, Preprint: \url{https://arxiv.org/pdf/1909.09575.pdf}.

\bibitem{BMdOS22}
W.~Barrera, L.~{Montes de Oca}, and D.A. Solis.
\newblock Comparison theorems for lorentzian length spaces with lower timelike
  curvature bounds.
\newblock \url{https://arxiv.org/pdf/2204.09612.pdf}, 2022.

\bibitem{Ber20}
T.~Beran.
\newblock Lorentzian length spaces.
\newblock Master's thesis, University of Vienna, 2020.
\newblock \url{https://phaidra.univie.ac.at/open/o:1363059}.

\bibitem{BR22}
T.~Beran and F.~Rott.
\newblock Gluing constructions for {L}orentzian length spaces.
\newblock \url{https://arxiv.org/pdf/2201.09695.pdf}, 2022.

\bibitem{BS22}
T.~Beran and C.~S{\"a}mann.
\newblock Hyperbolic angles in lorentzian length spaces and timelike curvature
  bounds.
\newblock \url{https://arxiv.org/pdf/2204.09491.pdf}, 2022.

\bibitem{BH99}
M.~R. Bridson and A.~Haeflinger.
\newblock {\em Metric spaces of non-positive curvature}, volume 319 of {\em
  Comprehensive Studies in Mathematics}.
\newblock Springer, Berlin, 1999.

\bibitem{BBI}
D.~Burago, Y.~Burago, and S.~Ivanov.
\newblock {\em A course in metric geometry}, volume~33 of {\em Graduate Studies
  in Mathematics}.
\newblock American Mathematical Society, Providence, {RI}, 2001.

\bibitem{BGH21}
A.~Burtscher and L.~García-Heveling.
\newblock Time functions on {L}orentzian length spaces.
\newblock \url{https://arxiv.org/pdf/2108.02693.pdf}, 2021.

\bibitem{CM20}
F.~Cavalletti and A.~Mondino.
\newblock {Optimal transport in {L}orentzian synthetic spaces, synthetic
  timelike {R}icci curvature lower bounds and applications}.
\newblock \url{https://arxiv.org/pdf/2004.08934.pdf}, 2020.

\bibitem{Ger70}
R.~Geroch.
\newblock Domain of dependence.
\newblock {\em Journal of Mathematical Physics}, 11(2):437--449, 1970.

\bibitem{GKS19}
J.~D.~E. Grant, M.~Kunzinger, and C.~S{ä}mann.
\newblock {Inextendibility of spacetimes and {L}orentzian length spaces}.
\newblock {\em Annals of Global Analysis and Geometry}, 55(1):133--147, 2019.

\bibitem{BH69}
S.W. Hawking and H.~Bondi.
\newblock The existence of cosmic time functions.
\newblock {\em Proceedings of the Royal Society of London. Series A.
  Mathematical and Physical Sciences}, 308(1494):433--435, 1969.

\bibitem{EH73}
S.W. Hawking and G.~F.~R. Ellis.
\newblock {\em The Large Scale Structure of Space-Time}.
\newblock Cambridge Monographs on Mathematical Physics. Cambridge University
  Press, 1973.

\bibitem{Hus18}
F.~Husseinov.
\newblock Extension of strictly monotonic functions in order-separable spaces.
\newblock \url{http://dx.doi.org/10.2139/ssrn.3260586}, 2018.

\bibitem{KS18}
M.~Kunzinger and C.~S{ä}mann.
\newblock {Lorentzian length spaces}.
\newblock {\em Annals of Global Analysis and Geometry}, 54(3):399--447, 2018.

\bibitem{KS21}
M.~Kunzinger and R.~Steinbauer.
\newblock Null distance and convergence of {L}orentzian length spaces.
\newblock \url{https://arxiv.org/pdf/2106.05393.pdf}, 2021.

\bibitem{Lev83}
V.L. Levin.
\newblock Continuous utility theorem for closed preorders on a metrizable $
  \sigma$-compact space.
\newblock {\em Dokl. Akad. Nauk SSSR}, 273:800--804, 1983.

\bibitem{MS21}
R.J. McCann and C.~S{\"a}mann.
\newblock A {L}orentzian analog for {H}ausdorff dimension and measure.
\newblock \url{https://arxiv.org/pdf/2110.04386.pdf}, 2021.

\bibitem{Min10}
E.~Minguzzi.
\newblock Time functions as utilities.
\newblock {\em Comm. Math. Phys.}, 298(3):855--868, 2010.

\bibitem{Min13a}
E.~Minguzzi.
\newblock Normally preordered spaces and utilities.
\newblock {\em Order 30}, pages 137--150, 2013.

\bibitem{Nac65}
L.~Nachbin.
\newblock {\em Topology and order}.
\newblock Number~4 in Vand Nostrand Mathematical Studies. Van Nostrand Co.,
  Inc., 1965.

\bibitem{ON83}
B.~O'Neill.
\newblock {\em Semi-Riemannian geometry with applications to relativity},
  volume 103 of {\em Pure and Applied Mathematics}.
\newblock Academic Press, Inc. [Harcourt Brace Jovanovich, Publishers], New
  York, 1983.

\bibitem{SV16}
C.~Sormani and C.~Vega.
\newblock Null distance on a spacetime.
\newblock {\em Classical {Q}uantum {G}ravity}, 33(8):085001, 2016.

\end{thebibliography}
\end{document}